\documentclass[oneside]{amsart}
\usepackage{graphicx}

\newtheorem{thm}{Theorem}[section]

\newtheorem{lem}[thm]{Lemma}
\newtheorem{prop}[thm]{Proposition}
\theoremstyle{definition}
\newtheorem{defn}[thm]{Definition}
\theoremstyle{remark}
\newtheorem{rem}[thm]{Remark}
\numberwithin{equation}{section}

\renewcommand{\t}{a}

\newcommand{\abs}[1]{\left\vert#1\right\vert}

\newcommand{\abslv}[1]{\abs{#1}_{\localF}}
\newcommand{\absvK}[1]{\abs{#1}_\K}

\newcommand{\set}[1]{\left\{#1\right\}}

\newcommand{\To}{\longrightarrow}

\newcommand{\abstinu}[1]{\abs{t_{#1}}_{\localF}^{\subII{s}{#1}{\nu}}}

\newcommand{\eqi}{}
\newcommand{\eqii}{}
\newcommand{\eqiii}{}
\newcommand{\eqiiii}{}
\newcommand{\eqiiiii}{}
\newcommand{\eqiiiiii}{}
\newcommand{\eqiiiiiii}{}
\newcommand{\eqiiiiiiii}{}

\newcommand{\eqix}{}
\newcommand{\eqx}{}

\newcommand{\comi}{}
\newcommand{\comii}{}
\newcommand{\comiii}{}
\newcommand{\comiiii}{}
\newcommand{\comiiiii}{}
\newcommand{\comiiiiii}{}
\newcommand{\comiiiiiii}{}
\newcommand{\comiiiiiiii}{}

\newcommand{\comix}{}
\newcommand{\comx}{}

\newcommand{\seteqi}[2][]{\renewcommand{\comi}{#1}\renewcommand{\eqi}{#2}}
\newcommand{\seteqii}[2][]{\renewcommand{\comii}{#1}\renewcommand{\eqii}{#2}}
\newcommand{\seteqiii}[2][]{\renewcommand{\comiii}{#1}\renewcommand{\eqiii}{#2}}
\newcommand{\seteqiiii}[2][]{\renewcommand{\comiiii}{#1}\renewcommand{\eqiiii}{#2}}
\newcommand{\seteqiiiii}[2][]{\renewcommand{\comiiiii}{#1}\renewcommand{\eqiiiii}{#2}}
\newcommand{\seteqiiiiii}[2][]{\renewcommand{\comiiiiii}{#1}\renewcommand{\eqiiiiii}{#2}}
\newcommand{\seteqiiiiiii}[2][]{\renewcommand{\comiiiiiii}{#1}\renewcommand{\eqiiiiiii}{#2}}
\newcommand{\seteqiiiiiiii}[2][]{\renewcommand{\comiiiiiiii}{#1}\renewcommand{\eqiiiiiiii}{#2}}

\newcommand{\seteqix}[2][]{\renewcommand{\comix}{#1}\renewcommand{\eqix}{#2}}
\newcommand{\seteqx}[2][]{\renewcommand{\comx}{#1}\renewcommand{\eqx}{#2}}

\newcommand{\stackeq}[1]{\stackrel{#1}{=}}
\newcommand{\eqcomii}{\stackeq{\comii}}

\newcommand{\Showeqiiiiii}{\begin{eqnarray*}
   \eqi & \eqcomii & \eqii \\
   \comiii & = & \eqiii \\
   \comiiii & = & \eqiiii \\
   \comiiiii & = & \eqiiiii \\
   \comiiiiii & = & \eqiiiiii
\end{eqnarray*}
}

\newcommand{\Showeqiiiiiii}{\begin{eqnarray*}
   \eqi & \eqcomii & \eqii \\
   \comiii & = & \eqiii \\
   \comiiii & = & \eqiiii \\
   \comiiiii & = & \eqiiiii \\
   \comiiiiii & = & \eqiiiiii \\
   \comiiiiiii & = & \eqiiiiiii
\end{eqnarray*}
}

\newcommand{\Showeqx}{\begin{eqnarray*}
   \eqi & \eqcomii & \eqii \\
   \comiii & = & \eqiii \\
   \comiiii & = & \eqiiii \\
   \comiiiii & = & \eqiiiii \\
   \comiiiiii & = & \eqiiiiii \\
   \comiiiiiii & = & \eqiiiiiii \\
   \comiiiiiiii & = & \eqiiiiiiii \\
   \comix & = & \eqix \\
   \comx & = & \eqx
\end{eqnarray*}
}

\newcommand{\showeqiii}[1]{\begin{eqnarray*}
   &&\hspace{#1}\eqi \\
   \comii & = &\eqii \\
   \comiii & = &\eqiii \\
\end{eqnarray*}
}

\newcommand{\showeqiiiii}[1]{\begin{eqnarray*}
   &&\hspace{#1}\eqi \\
   \comii & = &\eqii \\
   \comiii & = &\eqiii \\
   \comiiii & = &\eqiiii \\
   \comiiiii & = &\eqiiiii \\
\end{eqnarray*}
}
\newcommand{\showeqiiiiii}[1]{\begin{eqnarray*}
   &&\hspace{#1}\eqi \\
   \comii & = &\eqii \\
   \comiii & = &\eqiii \\
   \comiiii & = &\eqiiii \\
   \comiiiii & = &\eqiiiii \\
   \comiiiiii & = &\eqiiiiii \\
\end{eqnarray*}
}
\newcommand{\showeqiiiiiii}[1]{\begin{eqnarray*}
   &&\hspace{#1}\eqi \\
   \comii & = &\eqii \\
   \comiii & = &\eqiii \\
   \comiiii & = &\eqiiii \\
   \comiiiii & = &\eqiiiii \\
   \comiiiiii & = &\eqiiiiii \\
   \comiiiiiii & = &\eqiiiiiii \\
\end{eqnarray*}
}

\newcommand{\showeqx}[1]{\begin{eqnarray*}
   &&\hspace{#1}\eqi \\
   \comii & = &\eqii \\
   \comiii & = &\eqiii \\
   \comiiii & = &\eqiiii \\
   \comiiiii & = &\eqiiiii \\
   \comiiiiii & = &\eqiiiiii \\
   \comiiiiiii & = &\eqiiiiiii \\
   \comiiiiiiii & = &\eqiiiiiiii \\
   \comix & = &\eqix \\
   \comx & = &\eqx \\
\end{eqnarray*}
}

\newcommand{\s}{\textbf{s}}

\newcommand{\rank}{\tn{rank}}
\newcommand{\Span}{\tn{Span}}
\newcommand{\SpanF}{\tn{Span}_\F}
\newcommand{\dimF}{\dim_\F}
\newcommand{\meas}{\tn{meas}}

\newcommand{\TT}{\nP}
\newcommand{\TTnuo}{\TT_{\nuo}}

\newcommand{\muo}{\mu_{0,\nu}}
\newcommand{\dog}{dog}
\newcommand{\cat}{cat}

\newcommand{\ThetaNcusp}[2]{\Theta_{\psi}^{(#1)}(#2)}
\newcommand{\ThetaNcuspT}[2]{\Theta_{\psi,\nT}^{(#1)}(#2)}

\newcommand{\Spiiwave}[2]{\widetilde{\tn{Sp}}_{#1}(#2)}

\newcommand{\SpcartOA}[2]{\tn{Sp}_{#1}(\nW)_\fA\times\nO_{#2}(\nV)_\fA}

\newcommand{\Ddots}{\hskip2.2pt\raisebox{-1ex}{$\cdot$}\hskip-1pt\cdot\hskip-1pt\raisebox{1ex}{$\cdot$}}
\newcommand{\transpose}[1]{\hskip0.1pt ^t\hskip-0.3pt#1}
\newcommand{\tr}[1]{\transpose{#1}}

\newcommand{\Braceii}[4]{
   \left\{
   \begin{array}{lc}
     #1 & #2\\
     #3 & #4
   \end{array}
   \right.
}

\newcommand{\mativ}[4]{
\left(\begin{smallmatrix}
  #1 & #2  \\
  #3 & #4
\end{smallmatrix}\right)
}

\newcommand{\rr}[1]{\mathcal{R}(#1)}

\newcommand{\placeF}{\cP_{\F}}
\newcommand{\placeFinf}{\cP_{\F,\infty}}

\newcommand{\Sympform}[2]{\ll#1,#2\gg}
\newcommand{\sympform}[2]{<#1,#2>}
\newcommand{\orthform}[2]{(#1,#2)_{\nT}}
\newcommand{\wi}{w_1}
\newcommand{\wii}{w_2}
\newcommand{\vi}{v_1}
\newcommand{\vii}{v_2}
\newcommand{\viwi}{\vi\otimes\wi}

\newcommand{\viiwii}{\vii\otimes\wii}
\newcommand{\vw}{v\otimes w}

\newcommand{\supp}{\tn{supp}}

\newcommand{\stackit}[3][1]{\genfrac{[}{]}{0pt}{#1}{#2}{#3}}

\newcommand{\inbrackets}[1]{[#1]}
\newcommand{\str}[1]{#1^*}

\newcommand{\commdiagi}[6]{
\begin{array}{rcccl}
     & #1 & \stackrel{#3}{\longrightarrow} & #2 &   \\
  #5 & \downarrow &  & \downarrow & #6 \\
     & #1 & \stackrel{#4}{\longrightarrow} & #2 &
\end{array}
}

\newcommand{\ti}[2]{\abs{t_{#1}}^{#2}}
\newcommand{\locti}[2]{\abs{t_{#1}}^{#2}_{\local \F}}

\newcommand{\prodnu}[2]{\prod_{\nu#1}#2}

\newcommand{\secsize}{\Huge}
\newcommand{\MyTitle}{An Explicit Construction of CAP
Representations of $\GA$ associated to Non-trivial Characters of
Orthogonal Groups~$\HAnoT$}
\newcommand{\MyShortTitle}{CAP Representations of $\GA$
Lifted from a Character of the Orthogonal Group $\HA$}

\newcommand{\therefore}{\Longrightarrow}

\newcommand{\HALF}{\frac{1}{2}}
\newcommand{\mHALF}{\frac{m}{2}}

\newcommand{\rootunity}[1]{\mu_{#1}}
\newcommand{\rowvec}[2]{(#1,\ldots,#2)}
\newcommand{\colvec}[2]{\begin{pmatrix}
  #1 \\
  \vdots \\
  #2
\end{pmatrix}}

\newcommand{\unramchar}{\mu}
\newcommand{\unram}[1][\s]{\unramchar_{#1}}

\newcommand{\ftsection}[3][\normalsize]{{\section{#2}\label{#3}}}
\newcommand{\fatassletter}[1]{\mathbb{#1}}
\newcommand{\fA}{\fatassletter{A}}

\newcommand{\fC}{\fatassletter{C}}

\newcommand{\fF}{\fatassletter{F}}

\newcommand{\fK}{\fatassletter{K}}

\newcommand{\fN}{\fatassletter{N}}

\newcommand{\fR}{\fatassletter{R}}

\newcommand{\fT}{\fatassletter{T}}

\newcommand{\fW}{\fatassletter{W}}

\newcommand{\fZ}{\fatassletter{Z}}

\newcommand{\nA}{\tn{A}}
\newcommand{\nB}{\tn{B}}
\newcommand{\nC}{\tn{C}}
\newcommand{\nD}{\tn{D}}

\newcommand{\nF}{\tn{F}}
\newcommand{\nG}{\tn{G}}
\newcommand{\nH}{\tn{H}}
\newcommand{\nI}{\tn{I}}
\newcommand{\nJ}{\tn{J}}

\newcommand{\nM}{\tn{M}}
\newcommand{\nN}{\tn{N}}
\newcommand{\nO}{\tn{O}}
\newcommand{\nP}{\tn{P}}
\newcommand{\nQ}{\tn{Q}}

\newcommand{\nS}{\tn{S}}
\newcommand{\nT}{\tn{T}}

\newcommand{\nV}{\tn{V}}
\newcommand{\nW}{\tn{W}}
\newcommand{\nX}{\tn{X}}
\newcommand{\nY}{\tn{Y}}
\newcommand{\nZ}{\tn{Z}}

\newcommand{\nc}{\tn{c}}

\newcommand{\nf}{\tn{f}}

\newcommand{\nr}{\tn{r}}

\newcommand{\nt}{\tn{t}}

\newcommand{\nw}{\tn{w}}

\newcommand{\nz}{\tn{z}}

\newcommand{\B}{\tn{B}}

\newcommand{\coolletter}[1]{\mathcal{#1}}
\newcommand{\cA}{\coolletter{A}}
\newcommand{\cP}{\coolletter{P}}
\newcommand{\cO}{\coolletter{O}}

\newcommand{\gnot}{g_0}
\newcommand{\hnot}{h_0}
\newcommand{\Gi}{\nG_1}
\newcommand{\Gii}{\nG_2}
\newcommand{\Giwave}{\widetilde{\nG_1}}
\newcommand{\Giiwave}{\widetilde{\nG_2}}
\newcommand{\gi}{g_1}
\newcommand{\gii}{g_2}
\newcommand{\giii}{g_3}

\newcommand{\cocycle}[3][]{\nc_{#1}(#2,#3)}
\newcommand{\raococycle}[2]{\nc_{\tn{Rao}}(#1,#2)}

\newcommand{\splitting}[1][\nV]{\beta_{#1}}
\newcommand{\Splitting}[2][\nV]{\splitting[#1](#2)}

\newcommand{\SP}[1][2n]{\tn{Sp}_{#1}}

\newcommand{\PGSP}[1]{\tn{PGSp}(#1)}

\newcommand{\GL}[1]{\tn{GL}_{#1}}

\newcommand{\VB}{(\V,\B)}
\newcommand{\VBI}{(\V',\B')}

\newcommand{\go}{g_0}

\newcommand{\bfv}{\mbox{\boldmath$v$}}

\newcommand{\units}[1]{#1^\times}

\newcommand{\uF}{\units{\F}}
\newcommand{\uFsqr}{(\units{\F})^2}
\newcommand{\ufC}{\units{\fC}}

\newcommand{\thetakernel}[5][\phi]{\Sum{\bfv\in\VF{#2}}\funcIII{\Weil{#2}}{#3}{#4}{1}\func{#1}{#5}}

\newcommand{\thetakernelj}[6][\phi]{\Sum{#6}\funcIII{\Weil{#2}}{#3}{#4}{1}\func{#1}{#5}}

\newcommand{\thetakerneljjj}[6][\phi]{\Sum{#6}\funcII{\Weilj}{#3}{#4}\func{#1}{#5}}
\newcommand{\thetakerneli}[7][\phi]{\Sum{\bfv\in\VF{#2}}\funcIII{\Weil{#2}}{#3}{#4}{1}\funcIII{\Weil{#2}}{#6}{#7}{1}\func{#1}{#5}}
\newcommand{\thetakernelii}[6][\phi]{\Sum{\bfv\in\VF{#2}}#3\funcIII{\Weil{#2}}{#5}{#6}{1}\func{#1}{#4}}
\newcommand{\ACHAR}[2][]{\psi_{#1}(\half\tn{tr}(#2))}
\newcommand{\invACHAR}[2][]{\inv{\psi_{#1}}(\half\tn{tr}(#2))}
\newcommand{\achar}[1][\nC]{\psi_{#1}}
\newcommand{\Achar}[2][\nC]{\achar[#1](#2)}
\newcommand{\invAchar}[2][\nC]{\inv{\psi_{#1}}(#2)}
\newcommand{\Gr}{\tn{Gr}}

\newcommand{\F}{\fF}
\newcommand{\K}{\fK}

\newcommand{\Fstr}{\str{\F}}

\newcommand{\Lmana}[2]{#2 \backslash #1}
\newcommand{\Rmana}[2]{#1 / #2}

\newcommand{\PK}{\wp_\K}
\newcommand{\OF}{\cO_\F}
\newcommand{\OK}{\cO_\K}

\newcommand{\locOF}{\cO_{\localF}}

\newcommand{\genPK}{\bar{\omega}}

\newcommand{\ResFieldK}{\Rmana{\OK}{\PK}}

\newcommand{\SpWW}{\tn{Sp}(\fW)}
\newcommand{\SpW}{\tn{Sp}(\nW)}
\newcommand{\OV}{\nO(\nV)}
\newcommand{\SpWwave}{\widetilde{\tn{Sp}}(\nW)}

\newcommand{\SpWAwave}{\widetilde{\tn{Sp}_\fA}(\nW)}

\newcommand{\SpVWAwave}{\widetilde{\tn{Sp}_\fA}(\nV\tensor\nW)}
\newcommand{\SpWlocFwave}{\widetilde{\tn{Sp}}(\local{\nW})}
\newcommand{\SpWlocF}{\tn{Sp}(\local{\nW})}
\newcommand{\SpWA}{\tn{Sp}_\fA(\nW)}
\newcommand{\OVA}{\tn{O}_\fA(\nV)}
\newcommand{\RaoRep}[1]{\tn{r}(#1)}

\newcommand{\Hgp}{\mathbf{H}(\nW)}

\newcommand{\sub}[2]{#1_{#2}}
\newcommand{\subII}[3]{\sub{#1}{#2,#3}}

\newcommand{\map}[3]{#1~\colon~#2~\To~#3}
\newcommand{\embedmap}[3]{#1~\colon~#2~\hookrightarrow~#3}
\newcommand{\isomap}[3]{#1~\colon~#2~\stackrel{\cong}{\To}~#3}

\newcommand{\eqdef}{\stackrel{\tn{def}}{=}}

\newcommand{\defset}[2]{ \left \{ #1\mid#2 \right \}}

\newcommand{\Gadele}[1]{\sub{#1}{\fA}}
\newcommand{\GAdele}[1]{#1(\fA)}
\newcommand{\func}[2]{#1(#2)}
\newcommand{\funcII}[3]{#1\pair{#2}{#3}}
\newcommand{\funcIII}[4]{#1\orderedIII{#2}{#3}{#4}}

\newcommand{\tn}[1]{\textnormal{#1}}

\newcommand{\HH}[1][2n]{\tn{O}_{#1,\tn{T}}}
\newcommand{\GG}[1][4n]{\tn{Sp}_{#1}}

\newcommand{\Adelemana}[1]{ \Lmana {\GAdele{#1}}  {\func{#1}{\F}}  }
\newcommand{\HAoverHF}{\Adelemana{\HH}}
\newcommand{\GAoverGF}{\Adelemana{\GG}}
\newcommand{\ZAoverZF}{\Adelemana{\nZ}}

\newcommand{\loc}[1]{#1_{\localF}}
\newcommand{\MF}[1]{\nM_{#1}(\F)}
\newcommand{\HA}[1][2n]{\GAdele{\HH[#1]}}
\newcommand{\ZA}{\GAdele{\nZ}}
\newcommand{\HAnoT}{\GAdele{\tn{O}_{2n}}}
\newcommand{\GA}[1][4n]{\GAdele{\GG[#1]}}
\newcommand{\HF}{\HH(\F)}
\newcommand{\HFm}{\HH[m](\F)}

\newcommand{\locHF}{\HH(\localF)}
\newcommand{\locHFo}{\HH(\localo{\F})}
\newcommand{\locGF}{\GG(\localF)}
\newcommand{\locGFo}{\GG(\localo{\F})}
\newcommand{\locBF}{\nB(\localF)}

\newcommand{\IN}{\subseteq}
\newcommand{\INarr}{\hookrightarrow}
\newcommand{\isin}[2]{#1\in#2}

\newcommand{\pair}[2]{(#1,#2)}

\newcommand{\orderedIII}[3]{(#1,#2,#3)}

\newcommand{\inv}[1]{#1^{-1}}

\newcommand{\irr}[1]{\tn{Irr}(#1)}

\newcommand{\V}[1][]{\sub{\tn{V}}{#1}}
\newcommand{\VF}[1]{\V[\F]^{#1}}

\newcommand{\Int}[1]{\int_{#1}}
\newcommand{\half}{\tfrac{1}{2}}
\newcommand{\D}[2][]{\tn{d}_{#1}#2}
\newcommand{\XInt}[3]{\Int{#1}#2\D{#3}}
\newcommand{\locXInt}[3]{\Int{#1}#2\D[\nu]{#3}}
\newcommand{\ZInt}[1]{\Int{\ZAoverZF}#1\D{z}}
\newcommand{\HInt}[1]{\Int{\HAoverHF}#1\D{h}}

\newcommand{\thetafunc}[2][\phi]{\theta_{\psi}^{(#2)#1}}
\newcommand{\thetafuncj}[2][\phi]{\theta_{\psi}^{#1}}
\newcommand{\thetafuncT}[2][\phi]{\theta_{\psi,\nT}^{(#2)#1}}

\newcommand{\thetaintT}[1][\phi]{\Int{\HAoverHF}\funcIII{\thetafuncT[#1]{2n}}{g}{h}{1}\func{\OurChar}{h}\D{h}}
\newcommand{\thetaintkT}[2][\phi]{\Int{\HAoverHF}\funcIII{\thetafuncT[#1]{#2}}{g}{h}{1}\func{\OurChar}{h}\D{h}}
\newcommand{\genthetaintT}[4][\phi]{\Int{\HAoverHF}\funcIII{\thetafuncT[#1]{2n}}{#2}{#3}{1}\func{\OurChar}{#4}\D{h}}

\newcommand{\Sum}[1]{\sub{\sum}{#1}}

\newcommand{\locWeil}[1]{\omega_{\localpsi,\nT}^{(#1)}}
\newcommand{\locWeilo}[1]{\omega_{\localo{\psi},\nT}^{(#1)}}

\newcommand{\weilj}{\omega_{\psi}}
\newcommand{\Weil}[1]{\omega_{\psi,\nT}^{(#1)}}
\newcommand{\Weilj}{\omega_{\psi}}

\newcommand{\glbSchwartz}{\nS(\nV_{\fA}^{2n})}
\newcommand{\glbSchwartzk}[1]{\nS(\nV_{\fA}^{#1})}

\newcommand{\OurChar}{\xi}
\newcommand{\locOurChar}{\local{\OurChar}}
\newcommand{\locOurCharo}{\localo{\OurChar}}

\newcommand{\OurDualPair}[1][4n]{(\HA,\GA[#1])}
\newcommand{\OurDualPairk}[1]{(\HA,\GAk{k})}

\newcommand{\local}[1]{\sub{#1}{\nu}}
\newcommand{\localo}[1]{\sub{#1}{\nu_0}}

\newcommand{\nuo}{\nu_0}
\newcommand{\localpihat}{\local{\pihat}}
\newcommand{\localpihato}{\localo{\pihat}}
\newcommand{\localpi}{\local{\pi}}
\newcommand{\localpsi}{\local{\psi}}
\newcommand{\localphi}{\local{\phi}}
\newcommand{\localphinot}{\local{\phi}^0}
\newcommand{\localmu}{\local{\mu}}
\newcommand{\localsmu}{\mu_{\local{\s}}}
\newcommand{\localtau}{\local{\tau}}
\newcommand{\localF}{\local{\F}}
\newcommand{\localB}{\local{\nB}}
\newcommand{\localoF}{\localo{\F}}

\newcommand{\pivarphi}{\varphi_{\pi}}
\newcommand{\pivarphibar}{\overline{\varphi_{\pi}}}
\newcommand{\varphibar}{\overline{\varphi}}

\newcommand{\updownmiddle}[4]{#1^{#2}_{#3}#4}
\newcommand{\Ind}[3]{\updownmiddle{\textnormal{Ind}}{#1}{#2}{#3}}

\newcommand{\jordholder}[2]{\onetwon{\subseteq} {\V[0]} {\V[1]}
{\V[n]}}

\newcommand{\exactseqi}[7]{#1\longrightarrow#2\stackrel{#6}{\longrightarrow}#3\stackrel{#7}{\longrightarrow}#4\longrightarrow#5}

\newcommand{\onetwon}[4]{#2#1#3#1\cdots#1#4}
\newcommand{\onen}[3]{#2#1\cdots#1#3}

\newcommand{\hilbsymbI}[3]{\pair{#1}{#2}_{#3}}
\newcommand{\lochilbsymb}[2]{\hilbsymbI{#1}{#2}{\local{\F}}}
\newcommand{\hilbsymb}[2]{\hilbsymbI{#1}{#2}{\F}}
\newcommand{\hasse}{\tn{h}_\F}

\newcommand{\INPROD}[2]{\Int{\GAoverGF}#1\overline{#2}\D{g}}

\newcommand{\chglinespacing}[1]{}

\newcommand{\diag}{\tn{diag}}
\newcommand{\aut}{\tn{Aut}}
\newcommand{\pihat}{\widehat{\pi}}

\newcommand{\pibar}{\overline{\pi}}
\newcommand{\LtwoGAF}{\LtwoGAF}

\newcommand{\Vpihat}{V_{\widehat{\pi}}}

\newcommand{\Hom}[3]{\funcII{\sub{\tn{Hom}}{#1}}{#2}{#3}}
\newcommand{\cardprod}[2]{#1\times#2}

\newcommand{\textdf}[1]{\textbf{#1}}

\def\tensorF{\otimes_{\F}}
\def\tensor{\otimes}
\def\restensor{\otimes'}

\begin{document}

\chglinespacing{1.5}

\begin{titlepage}



\thispagestyle{empty}
\end{titlepage}




\pagenumbering{roman}
\newpage

\newpage


\newpage


\begin{center}
\chglinespacing{0.5} \MyTitle
\end{center}
\vspace{0.25cm}
\begin{center}
Ron Erez
\end{center}
\vspace{0.25cm}
\subsection*{Abstract}
 {\small In \cite{PS}, Piatetski-Shapiro the concept of CAP representations was introduced, elucidating the Saito-Kurokawa representations of $\PGSP{4}$. In this paper we present a family of CAP representations for the group $\GA$ through the application of the theta correspondence and Howe duality to the reductive dual pair $\OurDualPair$. The construction involves constructing a non-trivial automorphic character of $\HA$, lifting it to an irreducible cuspidal automorphic representation $\pi=\restensor_{\nu}\localpi$ of $\GA$, and providing a detailed characterization of the representations $\localpi$ of $\locGF$ at almost all places $\nu$.}

\vspace{1cm}
\noindent
\small{\emph{Keywords:} Theta lift; Automorphic Characters; CAP Representations}

\tableofcontents

\ftsection[\secsize]{Introduction}{sec:intro} 

\chglinespacing{1.5}In this work we will be concerned with a
certain subclass of cuspidal automorphic representations, that is,
our main concern will be \textbf{C}uspidal representations
\textbf{A}ssociated to a
proper \textbf{P}arabolic subgroup, i.e. CAP representations. 

\begin{defn}
Let G be a reductive group defined over a number field $\F$. Let P
be a proper parabolic subgroup of G defined over $\F$ with Levi
decomposition
$$\nP=\nM\nN$$ where $\nM$ is the Levi subgroup of P and $\nN$ is
the unipotent radical of P. Let
$\tau~\cong~\restensor_{\nu}\localtau$ be an irreducible,
automorphic, cuspidal representation of $\Gadele{\nM}$. Extend
$\tau$ to $\Gadele{\nP}~=~\Gadele{\nM}\Gadele{\nN}$ by letting
$\Gadele{\nN}$ act trivially. Then we say that an irreducible,
automorphic, cuspidal representation
$\pi\cong\restensor_{\nu}\localpi$ of $\Gadele{\nG}$ is a
\textdf{CAP representation} with respect to $\nP$ and $\tau$ if
for almost all places $\nu$, $\localpi$ is a constituent of
$\Ind{\loc{\nG}}{\loc{\nP}}{\localtau}$, i.e.
$$\localpi\cong\tn{ an irreducible subqoutient of }\Ind{\loc{\nG}}{\loc{\nP}}{\localtau}$$
\end{defn}

The notion of a CAP representation was first introduced by
Piatetski-Shapiro in \cite{PS}. Such cuspidal automorphic
representations apparently seem to contradict the generalized
Ramanujan-Petersson conjecture. The first examples of CAP
representations were given for the group $\tn{GSp}_4(\fA)$, by
Piatetski-Shapiro, in \cite{PS} where he reinterpreted the
Saito-Kurokawa representations. Later on, Soudry completed the
classification of CAP representations of
$\tn{GSp}_4(\fA)$~\cite{S2}.

\par
It may seem that these representations are rare, since it follows
from the work of Jacquet-Shalika \cite{JSH} that there are no CAP
representations of the general linear group. However, whether or not
such representations are`"rare" or "pathological" is open to debate
since much work has been done in this area. For instance, in
\cite{Ral-Sch1} and \cite{GGJ}, families of CAP automorphic
representations on the split exceptional group of type $G_2$ have
been constructed. In \cite{GJR}, a family of CAP representations of
a split group G of type $D_4$ is studied and constructed.

\par
In this work we construct a family of CAP representations of
$\GA$, which arise as theta lifts of non-trivial automorphic
characters of orthogonal groups $\HA$ associated to anisotropic
forms T.

\par
Here is a quick overview of this work. In
section \ref{sec:weil} we introduce the Weil representation and
recall its construction. In subsection \ref{subsec:Schrod} we
realize the Schr\"{o}dinger model of the Weil representation (Lemma
\ref{lem:WeilRep}). In subsection \ref{subsec:globalweil} we
describe the global Weil representation and describe a concrete
realization of $\SpWAwave$ the $2$-fold cover of the symplectic
group over the ad\`{e}le ring $\fA$ using Rao's normalized cocycle
$\raococycle{\cdot}{\cdot}$. Section \ref{sec:thetacorr} is a
recollection of preliminary concepts to be used in Section
\ref{sec:main}. In Subsection \ref{subsec:dualpair} we consider the
dual pair $(\nO(V),\tn{Sp}(W))$ and in Lemma \ref{lem:weilformulas},
we provide formulas for the restriction
$$\Weil{n}\eqdef\Weilj\mid_{\SpcartOA{2n}{2k}}\quad\INarr\quad\Spiiwave{4kn}{\nV\otimes\nW}_\fA$$
Subsections \ref{subsec:thetacorr} and \ref{subsec:howeduality}
provide a brief recap of the theta correspondence and Howe
duality, respectively.

\par
The main results will be proven in Section
\ref{sec:main} where we will construct CAP representations of the
symplectic group $\GA$, by lifting, via the theta-correspondence, a
non-trivial one-dimensional automorphic character $\OurChar$ of
$\HA$, an orthogonal group in $2n$ variables, corresponding to a
symmetric matrix T, over $\F$, which defines an anisotropic
quadratic form. Thus, by a result of reduction theory (See Theorem
\ref{TH:reductiontheory}), $\HAoverHF$ will be compact, therefore
simplifying our subsequent calculations and justifications of
calculations of Fourier coefficients, giving meaning to integrals,
exchanging summation with integration, etc.

\par
We start by considering a sequence of dual pairs
$\OurDualPair[2k]$, where $1\leq k\leq 2n$. By calculating Fourier
coefficients, we show that for every $1\leq k < 2n$, the
theta-lift $\ThetaNcusp{k}{\OurChar}$ is zero and for $k=2n$, it
is non-zero (Theorem \ref{TH:fourier}). It then follows from
results of Rallis (See Theorem \ref{TH:Rallis}) and Moeglin
\cite{M} that the theta lift $\pi\eqdef\ThetaNcusp{2n}{\OurChar}$
is an irreducible cuspidal representation of $\GA$ (Theorem
\ref{TH:IsCAP}). We then, by a result of Flath \cite{F}, may
express $\pi$ as a restricted tensor product
$$\pi\cong\tensor_{\nu}'\localpi$$
where almost all $\localpi$ are unramified representations of
$\locGF$ and in Theorem \ref{TH:ExplicitCon}, we give a detailed
description of the unramified factors $\localpi$ in terms of the
Hilbert symbol and show that $\pi$ is CAP (Theorem
\ref{TH:IsCAP}).

\pagestyle{myheadings}\markright{\MyShortTitle}

\vspace{0.5cm}
\begin{center}
    {\bfseries{Acknowledgements}}
\end{center}
    I extend my gratitude to Professor David Soudry for introducing me to the field of automorphic forms. This paper stems from the guidance provided in my thesis, overseen by Prof. David Soudry, who patiently led me through the process. I sincerely appreciate his continual support and patience. Any errors or inaccuracies in this work are solely my responsibility.
\chglinespacing{1.5} \pagenumbering{arabic} 

\ftsection[\secsize]{Notations and Preliminaries}{sec:back} 

\newcommand{\Jn}{\nJ_n}
\newcommand{\wn}{\nw_n}

\subsection{Notations} Let $\K$ be a locally compact,
nonarchimedean field with ring of integers $\OK$ and maximal ideal
of $\OK$, given by $\PK=\OK\genPK$, where $\genPK\in\PK-\PK^2$. Let
$q\eqdef\abs{\ResFieldK}$ denote the cardinality of the residue
field. Denote by $\absvK{\cdot}$ the absolute value of $\K$,
normalized so that $\abs{\genPK}=q^{-1}$. We fix a Haar measure
$\D{x}$ on $\K$ such that

$$\int_{\OK}\D{x}=1$$

and let

\newcommand{\DS}[1]{\tn{d}^{\times}x#1}

$$\DS=\frac{\D{x}}{\absvK{x}}$$
be the Haar measure on $\units{\K}$.
\par
Let $\fT\eqdef\defset{z\in\fC}{\abs{z}=1}$ and
$\rootunity{n}\IN\ufC$ denote the subgroup of $n$-th roots of unity.
Henceforth, whenever $\F$ is a number field,
we let $\fA=\fA_\F$ denote its ad\`{e}le ring.

\subsection{The Symplectic Group}\label{subsec:symp}
Let $\F$ be an arbitrary field ($\tn{char}(\F)\neq2$). The
symplectic group $\SP(\F)$ is the isometry group of a
non-degenerate, skew-symmetric, bilinear form on $\F^{2n}$. In
\cite[p. 19]{Gos} it is shown that there is "only one" symplectic
group $\SP(\F)$, for every $n\in\fN$, that is

\begin{prop}\label{PROP:GOS}
   Let $\VB,\VBI$ be two symplectic spaces of dimension 2n. Then there
   exists a symplectic isomorphism $\map{\sigma}{V}{V'}$.
\end{prop}

We will realize $\SP(\F)$ with respect to
$$\nJ_n=\mativ{}{\wn}{-\wn}{}$$
where
$$\wn=\begin{pmatrix}
        &        &  1    \\
        & \Ddots &       \\
    1   &        &
\end{pmatrix}$$

We let $\SP(\F)$ act on the right on the row space $\F^{2n}$

Then
$$\SP(\F)\eqdef\defset{\sigma\in\tn{GL}_{2n}(\F)}{\transpose{\sigma}\nJ_{2n}\sigma=\nJ_{2n}}$$

\begin{rem}
Recall that a symplectic basis
$\set{e_1,\ldots,e_n,e^*_1,\ldots,e^*_n}$ of $\F^{2n}$ (with
respect to $\sympform{x}{y}=x\nJ_{2n}\transpose{y}$) is a basis
such that
$$\sympform{e_i}{e_j}=\sympform{e_i^*}{e_j^*}=0,\quad\sympform{e_i}{e_j^*}=\delta_{ij},\quad 1\leq i,j\leq n$$
\end{rem}

We will be primarily concerned with the symplectic
group over the ad\`{e}le ring (of a number field) $\GA[2k]$ given by

$$\GA[2k]\eqdef\defset{\sigma\in\tn{GL}_{2k}(\fA)}{\transpose{\sigma}\nJ_{2k}\sigma=\nJ_{2k}}$$

where $\nJ_{2k}$ is identified with its diagonal embedding inside
$\tn{GL}_{2k}(\F)$.

\subsection{The Orthogonal group}\label{subsec:orth}

\par
The orthogonal group $\HFm$ is the isometry group of a
non-degenerate, symmetric, bilinear form, where the bilinear form is
given by the matrix $\nT\in\nM_m(\F)$, which is symmetric and
invertible. Unlike the symplectic case there is no equivalent to
Proposition \ref{PROP:GOS}, that is, there are "many" orthogonal
groups. Note that, for the sake of simplicity, we will primarily focus on orthogonal groups in
an even number of variables although our results do generalize to the odd case.

Thus, for a fixed symmetric, invertible, $2n\times 2n$-dimensional
matrix T over $\F$, we define
$$\HF\eqdef\defset{\sigma\in\tn{GL}_{2n}(\F)}{\transpose{\sigma}\nT\sigma=\nT}$$

Similarly, over the ad\`{e}le ring of a number field $\F$ we define
the orthogonal group, with matrix
$\nT=\transpose{\nT}\in\tn{GL}_{2n}(\F)$, by
$$\HA\eqdef\defset{\sigma\in\tn{GL}_{2n}(\fA)}{\transpose{\sigma}\nT\sigma=\nT}$$

\par
Without loss of generality we may assume that $\nT$ is a diagonal
matrix, since if $\gamma\in\GL{2n}(\F)$ is such that
$\gamma\cdot\nT\cdot\transpose{\gamma}=\nD$ is diagonal, then
$\inv{\gamma}\nO_{2n,\nD}(\localF)\gamma=\locHF$ at all places
$\nu$. Thus
$$\nT=\diag(\t_1,\ldots,\t_{2n}),\quad \t_i\in\uF,\quad1\leq i\leq 2n$$

Again, T is embedded diagonally inside $\tn{GL}_{2n}(\fA)$. Let
$\placeFinf$ be the finite set of Archimedean places of $\F$. We
will assume that $\F$ is such that at at least one place,
$\nuo\in\placeFinf$, the completion $\localoF$ is isomorphic to
$\fR$. Therefore, there exists an isomorphism
$$\isomap{\localo{\nI}}{\localoF}{\fR}$$
and we'll assume that $\t_i\in\F, 1\leq i\leq 2n$
are "positive" in $\localoF$, which means
$$\localo{\nI}(\t_i)>0,\quad\forall\t_i\in\F,\quad 1\leq i\leq 2n$$
\newcommand{\txsqr}[1]{\t_{#1}x^2_{#1}}
Alternatively, we may assume that the the $\t_i$ are "negative" in
$\localoF$, i.e.
$$\localo{\nI}(\t_i)<0,\quad\forall\t_i\in\F,\quad 1\leq i\leq2n$$

Therefore, over $\localoF$,
$$\txsqr{1}+\txsqr{2}+\cdots+\txsqr{2n}=0\quad\iff\quad x_{i}=0,\quad\forall 1\leq i\leq 2n$$
Now a necessary condition for T to represent zero over $\F$ is that
T represents zero at every place $\nu$, therefore T is anisotropic
over~$\F$.

\begin{rem}
   Note that if the homogenous space $\HAoverHF$ were compact then
   subsequent calculation justifications would be simplified. The
   following theorem shows that this is the case in our present
   setting, that is, where T is taken to be anisotropic over
   $\F$, for instance.
\end{rem}

\begin{thm}\label{TH:reductiontheory}
   $\HAoverHF$ is compact if and only if T is anisotropic over $\F$.

   \proof
   This is a result of reduction theory that can be found in \cite[Theorem 4.1.2, p.82]{We3}.
\end{thm}

\begin{rem}
   One might ask why couldn't we find a finite place such that T
   is anisotropic over the completion at that place, which would
   imply that T is anisotropic over $\fA$. The answer is that
   according to the following theorem, Theorem~\ref{TH:quadrepzero},
   if T has dimension $\geq5$ then it is isotropic.
\end{rem}

\begin{thm}\label{TH:quadrepzero}
   If Q is a quadratic form of rank $n\geq5$ over a local nonarchimedean field $\K$ then Q represents 0.
\end{thm}
\begin{proof}
   See \cite{Ser}.
\end{proof}

\subsection{Our Non-trivial Character}\label{subsec:setup}

Let $\F$ be a number field. We will construct a non-trivial
character $\OurChar$ of $\HAoverHF$. In section \ref{sec:main}, we
will lift this to a CAP representation of $\GA$.

\par
Let S be a non-empty finite set of places such that $\abs{\nS}$ is
even. In addition, let $\placeFinf\IN\nS$, i.e. S will contain the
(finite) set of Archimedean places. Then for every place $\nu$ of
$\F$ let $\locOurChar$ be the following character of $\locHF$: for
$\nu\in\nS$, $\locOurChar$ is the determinant character and for
$\nu\not\in\nS$, $\locOurChar$ is the trivial character. Define for
$h\in\HA$, $\OurChar(h)=\prod_\nu\locOurChar(\local h)$. This is a
character of $\HA$. It is trivial on $\HF$ since for $\isin{h}{\HF}$
embedded diagonally in $\HA$, i.e $h=(x,x,x,x,x,...)$, we have

\seteqi{\func{\OurChar}{h}}

\seteqii{\prodnu{}{\locOurChar(x)}}

\seteqiii{\prodnu{\in S}{\locOurChar(x)}\prodnu{\notin
S}{\locOurChar(x)}}

\seteqiiii{\prodnu{\in S}{\tn{det}(x)}\prodnu{\notin S}{1}}

\seteqiiiii{(\tn{det}(x))^{\abs{S}}}

\seteqiiiiii[\stackit{\tn{because
}\isin{\tn{det}(x)}{\set{\pm1}}}{\tn{and }\abs{S}\tn{ is even}}]
{1.}

\Showeqiiiiii

\subsection{Cocycles}\label{app:cocycle} 

\label{symb:cocycle} In this section, we recall some basic
definitions regarding cocycles. Let G be a locally compact group.

\begin{defn}\label{def:cocycle} A \textdf{cocycle} c on G (with values in $\ufC$) is a continuous map
$$\map{\nc}{\nG\times \nG}{\ufC},$$
such that for every $g,\gi,\gii,\giii\in\nG$
\begin{enumerate}
  \item
  $\cocycle{\gi\gii}{\giii}\cocycle{\gi}{\gii}=\cocycle{\gi}{\gii\giii}\cocycle{\gii}{\giii}$;
  \item $\cocycle{1}{g}=\cocycle{g}{1}=1$.
\end{enumerate}
\end{defn}

\begin{defn}
The cocycle c is called \textdf{trivial} if there exists a map
$\map{\splitting[]}{\nG}{\ufC}$ such that
$$\cocycle{\gi}{\gii}\Splitting[]{\gi\gii}=\Splitting[]{\gi}\Splitting[]{\gii}\qquad\tn{for all }\gi,\gii\in\nG$$
Such a map $\label{symb:splitting}\splitting[]$ is called a
\textdf{splitting} of c. If no such splitting exists then the
cocycle is said to be \textdf{nontrivial}.
\end{defn}

\begin{defn}
Let $\nc_1, \nc_2$ be two cocycles of G (with values in $\ufC$).
Then $\nc_1$ and $\nc_2$ are called \textdf{equivalent} if there
exists a $\map{\nt}{\nG}{\ufC}$ such that for every
$\gi,\gii\in\nG$ we have
$$\cocycle[1]{\gi}{\gii}=\cocycle[2]{\gi}{\gii}\frac{\nt(\gi)\nt(\gii)}{\nt(\gi\gii)}$$
\end{defn}

\subsection{The Hilbert Symbol and the Hasse Invariant}

Let $\F$ be a local field. Let $\rootunity{n}$ denote the
\label{symb:rtunity} $n^{\tn{th}}$ roots of unity subgroup of $\F$.
For proofs and generalizations see \cite{FV}, \cite{N} and
\cite{Ser}.

\label{symb:hilbsymb}
\begin{defn}
   Let $a,b\in\uF$. Then we define
   $\hilbsymb{a}{b}=1$ if and only if the equation
   $$z^2-ax^2-by^2=0$$ has a non-zero solution $(z,x,y)\in\F^3$,
   otherwise we define $\hilbsymb{a}{b}=-1$. The number
   $\hilbsymb{a}{b}\in\rootunity{2}$ is called the \textdf{Hilbert
   symbol} of $a$ and $b$ relative to $\F$.
\end{defn}

\noindent The following theorem outlines key properties of the Hilbert symbol.

\begin{thm}
Let $\F$ be a local field with $a,b,c\in\uF$. Then the Hilbert
symbol satisfies the following:
\begin{itemize}
  \item $\hilbsymb{ab}{c}=\hilbsymb{a}{c}\hilbsymb{b}{c}$
  \item $\hilbsymb{a}{b}=\hilbsymb{b}{a}$
  \item $\hilbsymb{a}{a}=\hilbsymb{a}{-1}$
  \item The Hilbert symbol is non-degenerate,\\ i.e if for a given
  $a\in\uF$, $\hilbsymb{a}{b}=1,\forall b\in\uF$, then $a\in\uFsqr$
  \item $\hilbsymbI{a}{b}{\fC}=1,\forall a,b\in\ufC$
  \item $\hilbsymb{a}{-a}=\hilbsymb{a}{1-a}=1,\forall a\in\uF,a\neq1$
  \item Let $\F$ be a non-archimedian local field. Then $\hilbsymb{x}{y}=1,\quad\forall x,y\in\units{\OF}$ where
  the characteristic of the residue field $\mathcal{K}_{\F}$ is
  not equal to 2.
\end{itemize}
\end{thm}
\noindent

\label{symb:hasse} 
\begin{defn}
   Let $\nQ(x_1,\ldots,x_n)=a_1x_1^2+\cdots+a_nx_n^2$ be a quadratic form.
   The \textdf{Hasse invariant} of $\nQ$ is given by
   $\hasse(\nQ)=\prod_{i<j}\hilbsymb{a_i}{a_j}\in\rootunity{2}$.\\
\end{defn}

\newpage
\ftsection[\secsize]{Weil Representation}{sec:weil} 

In this section we recall formulas for the Weil representation (see
\cite{Rao}) and state some of its properties.

\par
Let $\F$ be a local field, and $\psi$ a non-trivial character of
$\F$. We assume $\tn{char}(\F)\neq2$. Let W be a $2n$-dimensional
symplectic vector space over $\F$ with symplectic form $<,>$. Let
$\nG=\tn{Sp}_{2n}(\nW)$ be the symplectic group of $(\nW,<,>)$. We
let $\nG$ act on W from the right.

 There are many explicit models of the metaplectic representation over a local field. For various models see
\cite[p. 28-30]{MVW}. We will be concerned with the
Schr\"{o}dinger model of the Weil representation.

\newcommand{\xo}{x_0}
\newcommand{\rhorep}{\rho_{\psi}}
\subsection{The Schr\"{o}dinger model}\label{subsec:Schrod} Let
$\nW$ be a symplectic vector space over a local field $\F$. Let
$\nW=\nX\oplus\nY$ be a complete polarization of W. Let
$\nS=\nS(\nX)$ be the \textdf{Schwartz space} on X. Then by
\cite[p.7, Lemma 2.2]{Ku2} we may define the representation
$\map{\rhorep}{\Hgp}{\aut(\nS(\nX))}$ of $\Hgp$ by
$$(\rhorep((x+y,t))\varphi)(\xo) = \psi(\sympform{\xo}{y}+\HALF\sympform{x}{y}+t)\varphi(\xo+x)
\quad(x,\xo\in\nX,y\in\nY,t\in\F).$$
where $\Hgp$ is the Heisenberg group and
$\rhorep$ is irreducible and its central character is $\psi$. (see \cite{Ku2} and \cite[p. 28-30]{MVW} for more details).

\begin{defn}
   The representation of the group $\SpWwave$ in the space S(X) is
   called the \textdf{Schr\"{o}dinger model} of the Weil
   representation associated to the complete polarization
   $\nW=\nX\oplus\nY$.
\end{defn}

From \cite[p. 351, Lemma 3.2]{Rao} (or \cite{Ku2}) we have the
following explicit description of 
the Schr\"{o}dinger model.

\par
Let $$g=\mativ{a}{b}{c}{d}$$ be the matrix representation of
$g\in\SpW$ with respect to the complete polarization
$\nW=\nX\oplus\nY$ and recall that $\SpW$ acts on the right of
$\nW$. Thus, we have
$$a\in\Hom{}{\nX}{\nX},\quad
b\in\Hom{}{\nX}{\nY},\quad c\in\Hom{}{\nY}{\nX},\quad
d\in\Hom{}{\nY}{\nY}$$

\newcommand{\Yoverkerc}{\ker(c)\backslash\nY}
\renewcommand{\dog}{\HALF\sympform{xa}{xb}+\sympform{xb}{xc}+\HALF\sympform{yc}{yd}}
\renewcommand{\cat}{xa+yc}
\newcommand{\aastr}{\mativ{a}{}{}{\str{a}}}
\newcommand{\bbstr}{\mativ{b}{}{}{\str{b}}}
\newcommand{\unib}{\mativ{1}{b}{}{1}}
\newcommand{\uniy}{\mativ{1}{y}{}{1}}
\newcommand{\deta}{\abs{\det(a)}^\HALF}
\newcommand{\detone}{\abs{\det(1)}^\HALF}

\begin{lem}\label{lem:WeilRep}
   There is a unique choice of Haar measure $d\mu_g(y)$ on
   $\tn{ker}(c)\backslash\nY$, such that the operator $r(g)$
   defined for $\varphi\in\nS(\nX)$ by
   $$(\RaoRep{g}\varphi)(x)\eqdef\Int{\Yoverkerc}\psi(\dog)\cdot\varphi(\cat)d\mu_g(y)$$
   is an intertwining operator between $\rho_\psi$ and $\rho_\psi^g$ in
   $\nS(\nX)$. Moreover, when $c=0$, the intertwining operator
   $\RaoRep{\cdot}$ is given by
   $$(\RaoRep{\mativ{a}{}{}{\str{a}}}\varphi)(x)=\abs{\det(a)}^{\HALF}\varphi(xa),\qquad a\in GL(X)\qquad(*)$$
   and
   $$(\RaoRep{\mativ{1}{b}{}{1}}\varphi)(x)=\psi(\HALF\sympform{x}{xb})\cdot\varphi(x),\qquad b=\tr{b}, b\in\Hom{}{X}{Y}\qquad(**)$$
\end{lem}

\begin{proof}
Omitted.

\end{proof}

\subsection{The Metaplectic Group over the Ad\`{e}le Ring}\label{subsec:globalweil}
\newcommand{\SpWAhat}{\widehat{\tn{Sp}}_\fA(W)}
\newcommand{\locTwoCover}{\widetilde{Sp}({\local{W}})}
\newcommand{\loceps}{\local{\varepsilon}}

Let $\F$ be a number field. In \cite{Howe} Roger Howe describes the
two-fold central extension of the symplectic group over the
ad\`{e}le ring, $\SpWAwave$. In this subsection we give a concrete
realization of the group $\SpWAwave$. Let $W,<\cdot,\cdot>$ be a
symplectic vector space over $\F$ with $\dim_\F W=2n$ and fix a
choice of a symplectic basis
$B=\set{e_1,\ldots,e_n;e_1^*,\ldots,e_n^*}$ of $W$ over $\F$.
Consider the restricted direct product

$$\SpWAhat\eqdef\prod_\nu'\locTwoCover$$

\newcommand{\maxcomp}{\tn{Sp}(W_{\locOF})}
\noindent
with respect to $\defset{\loceps(\local K)}{\nu\tn{ odd and
finite}}$ where $\local{K}\eqdef\maxcomp$ is the maximal compact
subgroup which stabilizes the $\locOF$-lattice in $\local W$ spanned
by the symplectic basis $B$. It is known that there is an embedding
$\map{\loceps}{\local{K}}{\SpWlocFwave}$ such that
$\loceps(r)=(r,\local{\epsilon}(r))$ (see \cite[p. 43]{MVW}).

\par
This restricted direct product does not furnish us with a two-fold
central extension of $\SpWA$, since the kernel of the projection
homomorphism is too big. Thus it would seem reasonable to consider
the quotient of the restricted product
$\SpWAhat\eqdef\prod_\nu'\locTwoCover$ and a subgroup that will
cause the kernel to be "very small", i.e. to be isomorphic to
$\rootunity{2}$. More precisely, consider the subgroup

$$C'\eqdef\defset{\prod_\nu(I,\loceps)\in\SpWAhat}{\prod_\nu\loceps=1}.$$

Then the group
$$\SpWAwave\eqdef C'\backslash\SpWAhat$$
is a two-fold central extension of $\SpWA$ and we have an exact
sequence

$$\exactseqi{1}{\rootunity{2}\cong\set{C'\prod_nu(1,\loceps)\in\SpWAwave}}{\SpWAwave}{\SpWA}{1}{i}{p}$$

where

$$\map{p}{\SpWAwave}{\SpWA}$$

is the natural projection. Note that as in \cite{Howe}, we have the
commutative diagram

$$\begin{array}{rcccl}
     & \SpWlocFwave & \longrightarrow & \SpWAwave &   \\
     & \downarrow &  & \downarrow &  \\
     & \SpWlocF & \longrightarrow & \SpWA &
\end{array}
$$

where the vertical arrows are the natural projections. Let $\psi$
be a non-trivial additive character of $\F\backslash\fA$. Denote
$X=\Span\set{e_1,\ldots,e_n}$ and
$Y=\Span\set{e_1^*,\ldots,e_n^*}$ (regarded as algebraic groups
over $\F$). Then we have the global Weil representation $\weilj$
of $\SpWAwave$ acting on $S(X_\fA)$. It is given by
$\tensor'_\nu\omega_{\local{\psi}}$, where $\omega_{\local{\psi}}$
are the local Weil representations of $\SpWlocFwave$ acting in
$S(\local{X})$. Thus, let $\phi\in S(X_\fA)$ be a decomposable
function, i.e. $\phi(x)=\prod_\nu\local{\phi}(\local x)$, where
$x\in X_\fA, \local{\phi}\in S(\local{X})$ and at almost all
finite places $\nu$, $\local \phi$ is the characteristic function
$\localphinot$ of $\oplus_{i=1}^n\locOF e_i$. Let
$\tilde{g}=C'\prod_\nu(\local g, \local z)$ where $(\local g,
\local z)\in\SpWlocFwave, g=\prod_\nu\local{g}\in\SpWA$ and
$\prod_\nu(\local g, \local z)\in\SpWAhat$. Then
$\weilj(\tilde{g})\phi(x)=\prod_\nu\omega_{\local{\psi}}(\local g,
\local z)\localphi(\local x)$. (It is known that
$\omega_{\local{\psi}}(\loceps(\local
K))\localphinot=\localphinot$, for almost all finite places
$\nu$.) Finally, for $\gamma\in\tn{Sp}(W_\F)$ the map
$\gamma\mapsto C'\prod_\nu(\gamma,1)$ is an embedding of
$\tn{Sp}(W_\F)$ inside $\SpWAwave$.


\ftsection[\secsize]{The Theta Correspondence and Reductive Dual
Pairs}{sec:thetacorr} 
\label{symb:thetadistrib}
\label{subsec:dualpair}

\newcommand{\cent}[2]{\tn{Cent}_{#1}(#2)}
We recall the Weil representations, recall the theta correspondence and Howe duality.
\par 
Let $\F$ be a field ($\tn{char}(\F)\neq2$). Let $V$ (respectively $W$) be a finite
dimensional vector space over $\F$, equipped with a non-degenerate
symmetric (respectively anti-symmetric) bilinear form
$(\cdot,\cdot)$ (respectively $<\cdot,\cdot>$). Consider the
orthogonal group $\OV$ and the symplectic group $\SpW$. Suppose
$\dim_\F(\nV)=m$ and $\dim_\F(\nW)=2n$. Construct the
$2mn$-dimensional $\F$-vector space
$$\fW\eqdef\nV\tensorF\nW$$
This vector space becomes a symplectic vector space together with
the symplectic form given by
$$\Sympform{\viwi}{\viiwii}\eqdef\orthform{\vi}{\vii}\sympform{\wi}{\wii}$$
and we have an embedding (with a small kernel
$\defset{(\epsilon\nI_\nV,\epsilon\nI_\nW)}{\epsilon\in\set{1,-1}}$)
$$
\begin{array}{c}
  \OV\times\SpW\To\SpWW \\
  (v,w)\mapsto\vw
\end{array}
$$

\par
The pair $(\OV,\SpW)$ is a reductive dual pair. Let $\F$ be a local
field. Then the symplectic group $\SpW$ has a unique (non-linear)
double cover $\SpWwave$ (unless $\F=\fC$). This double cover always
splits over the subgroup $\OV$ and splits over $\SpW$ if and only if
$\dim(\nV)=m$ is even.
\par
Thus if $m$ is even we have an
embedding
$$\fZ_2\backslash\OV\times\SpW\hookrightarrow\widetilde{\tn{Sp}}(\fW)$$
and if $m$ is odd we have an embedding
$$\OV\times\SpWwave\hookrightarrow\widetilde{\tn{Sp}}(\fW)$$

\subsection{The Weil Representation Restricted to $(\HA[2k],\GA[2n])$}
In this subsection, $\F$ is a number field $V$ will be the
$m$-dimensional column space over $\F$. We assume that $m$ is even.
Consider a non-degenerate, symmetric bilinear form
$(\cdot,\cdot)_\nT$ on $V\times V$, where $\nT$ is constructed as in
subsection \ref{subsec:orth}. Let W be a symplectic vector space of
dimension $\dim_\F\nW=2n$ over $\F$ for some $n\in\fN$ with
symplectic form $<\cdot,\cdot>$.

\par
We keep the notation of subsection \ref{subsec:globalweil}. Let
$\Weilj$ be the Weil representation of $\SpVWAwave$. Let
$$\Weil{n}\eqdef\Weilj|_{\SpWA\times\OVA}.$$

Let $\nW=\nX\oplus\nY$ be a complete polarization of $\nW$ over
$\F$. (We write the elements of $\SpW$ as in subsection
\ref{subsec:symp}.) Then $\Weil{n}$ acts on
$$\nS(\nX\otimes\nV)_\fA\cong\nS(\nV^{n}_\fA)$$
because $\nW\otimes\nV=\nX\otimes\nV\oplus\nY\otimes\nV$ is a
complete polarization of $\nW\otimes\nV$ and
$$\nX\otimes\nV\cong\underbrace{\nV\oplus\cdots\oplus\nV}_{\tn{n times}}=\nV^{n}$$

Before we proceed any further, we would like to have some formulas
for the Weil representation restricted to our dual pair. The
following lemma furnishes us with formulas for $\Weil{n}$ which will be used throughout.

\begin{lem}\label{lem:weilformulas}

   Assume $\dim_\F\nW=2n$ and $\dim_\F\nV=m$, where $m$ is even. Then for $\phi\in\nS(\nV^n_\fA),x\in\nV^n_\fA$
   we have the following formulas
   \begin{equation}\label{weil:i}
    \Weil{n}(\aastr,1)\phi(x)=\hilbsymbI{\det(a)}{(-1)^{\frac{m}{2}}\cdot\det(\nV)}{}\abs{\det(a)}^{\mHALF}\phi(xa)
\end{equation}
   where $a\in\GL{n}(\fA)$ and $\det\nV\eqdef\det\nT$ and
   $\hilbsymbI{\cdot}{\cdot}{}$ denotes the product over all places of $\F$ of the local Hilbert symbols and
   $xa=(v_1,\ldots,v_n)\cdot a$.
   and

   \begin{equation}\label{weil:ii}
    \Weil{n}(\unib,1)\phi(x)=\ACHAR{\tn{Gr}(x)\wn b}\cdot\phi(x)
   \end{equation}

   where $\wn b=\tr{(\wn b)},\ \tn{Gr}(x)\eqdef(\orthform{x_i}{x_j})_{i,j}$

   \begin{equation}\label{weil:iii}
   \Weil{n}(1,h)\phi(x)=\phi(\inv{h}\cdot x)
   \end{equation}

\end{lem}
\newcommand{\jV}{\tn{j}_{\nV}}
\newcommand{\jVgi}{\jV(\gi)}
\newcommand{\jVgii}{\jV(\gii)}
\newcommand{\SpVW}{\tn{Sp}(\nV\otimes\nW)}
\newcommand{\jVgigii}{\jV(\gi\gii)}

\newcommand{\jW}{\tn{j}_{\nW}}

\begin{proof}
   This follows from \cite[Proposition 4.3, p. 37]{Ku2} for $m$
   even.
\end{proof}

\begin{rem}
   The formulae in Lemma \ref{lem:weilformulas} are valid at each
   place, i.e. they are valid locally.
\end{rem}


\subsection{Theta Correspondence}\label{subsec:thetacorr}
Define a
linear functional
$$\map{\Theta}{\nS(\nX(\fA))}{\fC}\qquad\qquad(*)$$
given by
$$\Theta(\phi)=\Sum{x\in\nX(\F)}\phi(x),\quad\phi\in\nS(X(\fA))$$

This distribution is $\tn{Sp}(W)(\F)$-invariant, i.e.
$\Theta(\omega_\psi(\gamma)\phi)=\Theta(\phi)$ for all
$\gamma\in\tn{Sp}(V\tensor W)(\F)$. This is a fundamental theorem of
Weil.


Given any function $\phi\in\nS(\nX(\fA))$, the following formula
defines a function $\thetafuncj{2n}$ on
$\SpW(\F)\backslash\SpWwave(\fA)$ given by
$\thetafuncj{2n}$\label{symb:thetafunc}
\begin{eqnarray*}
   \thetafuncj{2n}(g,z) & \eqdef& \Theta(\Weilj(g,z)\phi)\\
                     & =     & \Sum{x\in\nX(\F)}(\Weilj(g,z)\phi)(x)
\end{eqnarray*}

Given a reductive dual pair $(\Gi,\Gii)$ inside $\SpWW$, Howe's idea
was to restrict the Weil representation on $\SpWwave(\fA)$ to the
product $\widetilde{\Gi(\fA)\cdot\Gii(\fA)}$. Now, $\Giwave(\fA)$ is
the inverse image of $\Gi(\fA)\times\nI$ and similarly,
$\Giiwave(\fA)$ is the inverse image of $\nI\times\Gii(\fA)$. There
is a homomorphism
$\Giwave(\fA)\times\Giiwave(\fA)\mapsto\widetilde{\Gi(\fA)\cdot\Gii(\fA)}$
with kernel
$\defset{(1,\epsilon),(\epsilon,1)}{\epsilon\in\rootunity{2}}$. Thus
by restricting to $\Giwave(\fA)\times\Giiwave(\fA)$, we get a
function on $\Giwave(\fA)\times\Giiwave(\fA)$ called the
\textdf{theta-kernel}, given by
$$\thetafuncj{2n}(g,h)=\thetakerneljjj{2n}{g}{h}{x}{x\in\nX(\nF)}$$
\par
Let $\cA_1$ be an irreducible space of cusp forms on
$\Gi(\F)\backslash\Giwave(\fA)$. Fix an $\nf\in\cA_1$. We define a
function
$$\gii\mapsto\Int{\Gi(\F)\backslash\Giwave(\fA)}\thetafuncj{2n}(\gi,\gii)\nf(\gi)d\gi$$
on $\Gii(\F)\backslash\Giiwave(\fA)$. The span of these functions
when $\phi$ runs over $\nS(\nX(\fA))$ is a $\Giiwave(\fA)$-invariant
space of functions. Moreover, if in addition, we take the span of
all $\nf$ that run over $\cA_1$ then we obtain the so-called
\textdf{theta lift} of the cuspidal representation $\cA_1$ to $G_2$
and is denoted by $\Theta_{\psi}(\cA_1)$.\label{symb:thetalift}
Moreover, convergence of the above integral is gauranteed by the
fact that cusp forms $\nf(\gi)$ are known to be rapidly decreasing
(in a Siegel domain) while $\thetafuncj{2n}(\gi,\gii,1)$ is of
moderate growth in $\gi$ and in $\gii$.
\par
Similarly, when the roles of $\Gi$ and $\Gii$ are reversed. Let
$\cA_2$ be an irreducible space of cusp forms on
$\Gii(\F)\backslash\Giiwave(\fA)$. Then we obtain the theta lift
of $\cA_2$ to $\Gi$ denoted by $\Theta_{\psi}(\cA_2)$.

\subsection{Howe Duality}\label{subsec:howeduality}

Let $\F$ be a local field. Let $\pair{\Gi}{\Gii}$ be a
reductive dual pair inside $\SpW$. Let $\Giwave,\Giiwave$ be inverse
images of $\Gi,\Gii$ over a local field $\F$, respectively.

\begin{defn} Given admissible representations $\isin{\pi}{\irr{\Giwave}},\isin{\sigma}{\irr{\Giiwave}}$.
$\pi,\sigma$ are said to \textdf{correspond} with respect to
$\omega_\psi$ if
$\Hom{\Giwave\times\Giiwave}{\omega_\psi}{\pi\tensor\sigma}\neq0$.
The correspondence between $\irr{\Giwave}$ and $\irr{\Giiwave}$ is
the set
$$\defset{\pair{\pi}{\sigma}\in\cardprod{\rr{\Giwave}}{\rr{\Giiwave}}}{\Hom{\Giwave\times
\Giiwave}{\omega_\psi }{\pi\tensor\sigma}\neq0}$$ where
$$\rr{\Giwave}\IN\irr{\Giwave},\rr{\Giiwave}\IN\irr{\Giiwave}$$ are given by
$$\rr{\Giwave}\eqdef\defset{\isin{\pi}{\irr{\Giwave}}}{\Hom{\tilde{\Gi}}{\omega_\psi }{\pi}\neq0}$$
$$\rr{\Giiwave}\eqdef\defset{\isin{\sigma}{\irr{\Giiwave}}}{\Hom{\tilde{\Gii}}{\omega_\psi }{\sigma}\neq0}$$
\end{defn}

Howe Duality implies that
\begin{enumerate}
  \item Every
  $\pi\in\rr{\Giwave}$ (or $\sigma\in\rr{\Giiwave}$) occurs in the
  correspondence.
  \item For all
  $\isin{\pi}{\irr{\Giwave}},\isin{\sigma_1,\sigma_2}{\irr{\Giiwave}}$\\
  $\Hom{\Giwave\times\Giiwave}{\omega_\psi }{\pi\tensor\sigma_i}\neq0,i=1,2$
  $\Rightarrow\sigma_1\cong\sigma_2$\\ For all
  $\isin{\sigma}{\irr{\Giiwave}},\isin{\pi_1,\pi_2}{\irr{\Giwave}}$\\
  $\Hom{\Giwave\times\Giiwave}{\omega_\psi }{\pi_i\tensor\sigma}\neq0,i=1,2$
  $\Rightarrow\pi_1\cong\pi_2$
\end{enumerate}

\begin{thm}\label{TH:Howeduality}
   Howe duality is true over all archimedean local fields (proved by
   Howe) and all $p$-adic fields with $p\neq2$ (proved by Waldspurger).
\end{thm}

\newpage
\ftsection[\secsize]{Main Results}{sec:main}
In this section, $\F$ is a number field, which has at least one real
place. The matrix $T$ denotes the anisotropic matrix described in subsection
\ref{subsec:orth} and $\map{\OurChar}{\HA}{\ufC}$ is a
non-trivial automorphic character of $\HA$ as described in subsection
\ref{subsec:setup}.

\subsection{Our Main Results}\label{subsec:main}
\newcommand{\Sch}{\mathcal{S}}


Let $\ThetaNcuspT{k}{\OurChar}$ denote the representation of
$\SP[2k](\fA)$ generated by
$$\defset { g\mapsto\thetaintkT{k}}{\phi\in\glbSchwartzk{k}},\quad 1\leq k\leq 2n $$


where
$$\thetafuncT{k}(g,h,1)=\thetakernel{k}{g}{h}{\bfv},\quad\phi\in\nS(\nV^{k}_\fA)$$
is a theta kernel. This is the theta lift of $\OurChar$ to
$\tn{Sp}_{2k}(\fA)$. Recall that in our setting $\HAoverHF$ is compact.

\newcommand{\vnot}{\mathbf{v^0}}


\par
We are now ready to formulate the main theorems to be proved.


\begin{thm}\label{TH:fourier}
Let $T$ be the anisotropic matrix from section \ref{subsec:orth}. Let
$\map{\OurChar}{\HA}{\ufC}$ be the non-trivial automorphic character of
$\HA$ as defined in subsection \ref{subsec:setup}. Then
    \begin{enumerate}
      \item $\ThetaNcuspT{2n}{\OurChar}\neq0$
      \item $\ThetaNcuspT{k}{\OurChar}=0,\quad\forall 1\leq k<2n$
   \end{enumerate}
\end{thm}

By applying this theorem together with Rallis' theorem (Theorem
\ref{TH:Rallis}) we almost immediately obtain the following result

\begin{thm}\label{TH:ThetaRepCuspidal}
   $\ThetaNcuspT{2n}{\OurChar}$ is an irreducible, automorphic,
   cuspidal representation of $\GA$.
\end{thm}

\newcommand{\unramnu}{\unram^{(\nu)}}
\newcommand{\chinu}{\chi^{(\nu)}}

All this will be done in subsection \ref{subsec:Fourier}. Finally,
the following main results (Theorems \ref{TH:ExplicitCon} and
\ref{TH:IsCAP}) will be proved in subsection
\ref{subsec:explicitcon}.
\begin{thm} \label{TH:ExplicitCon}
   For almost all $\nu$, the local component $\localpi$ of $\pi=\ThetaNcuspT{2n}{\OurChar}$
   is a constituent of $\Ind{\locGF}{\localB}{(\unramnu\cdot\chinu_\nT)}$ where
   $$\s=(-n,-n+1,-n+2,\ldots,-1,0,1,2,\dots,n-3,n-2,n-1)$$
   and $\unramnu$ is the unramified character of the standard Borel subgroup $\localB$ given by
   $$\unramnu\begin{pmatrix}
  t_1    &           &        &           &        &        \\
         & \ddots    &        & \mathbf{*}&        &        \\
         &           & t_{2n}    &           &        &        \\
         &           &        & t_{2n}^{-1}  &        &        \\
         & \mathbf{0}&        &           & \ddots &        \\
         &           &        &           &        & t_1^{-1}
\end{pmatrix}=\locti{1}{-n}\locti{2}{-n+1}\locti{3}{-n+2}\cdots\locti{2n}{n-1}$$
and $\chinu_\nT$ is given by
   $$\chinu_\nT\begin{pmatrix}
  t_1    &           &        &           &        &        \\
         & \ddots    &        & \mathbf{*}&        &        \\
         &           & t_{2n}    &           &        &        \\
         &           &        & t_{2n}^{-1}  &        &        \\
         & \mathbf{0}&        &           & \ddots &        \\
         &           &        &           &        & t_1^{-1}
\end{pmatrix}=\hilbsymbI{t_1 \cdots t_{2n}}{(-1)^n\det\nT}{\local \F}$$

\end{thm}
\begin{rem}
   Notice that in this particular theorem $\local{\s}=\s$ does not
   depend on $\nu$.
\end{rem}

\newcommand{\ufA}{\units{\fA}}

Denote by $\chi_\nT$ the character of $\uF\backslash\ufA$ given by
$\chi_\nT(x)=\prod_\nu\lochilbsymb{\local x}{(-1)^n\det
T}=\prod_\nu\chinu_\nT(\local x)$. Denote also
$\unram(x)=\prod_\nu\unramnu(\local x)$ for $x\in\ufA$. This is a
character of $\uF\backslash\ufA$.

\begin{thm} \label{TH:IsCAP}
   $\ThetaNcuspT{2n}{\OurChar}$ is a CAP representation of $\GA$
   with respect to B and $\tau$, where B is the Borel subgroup of $\GA$
   and $\tau=\unram\chi_\nT$.
\end{thm}

\subsection{Calculation of Fourier Coefficients} \label{subsec:Fourier}

Consider the orthogonal group $\HA$. Let $\sigma$ be a cuspidal
representation of $\HA$. Consider the following series of dual
pairs, together with their associated theta-lifts.


$$
\begin{array}{c@{\quad-\quad}c}
  \vdots & \vdots \\
  \OurDualPair[4n+2] & \ThetaNcusp{2n+1}{\sigma} \\
  \OurDualPair[4n] & \ThetaNcusp{2n}{\sigma} \\
  \OurDualPair[4n-2] & \ThetaNcusp{2n-1}{\sigma} \\
  \vdots & \vdots \\
  \OurDualPair[6] & \ThetaNcusp{3}{\sigma} \\
  \OurDualPair[4] & \ThetaNcusp{2}{\sigma} \\
  \OurDualPair[2] & \ThetaNcusp{1}{\sigma}
\end{array}
$$


\begin{thm}[Rallis]\label{TH:Rallis}
   Let $\sigma$ be an irreducible automorphic cuspidal representation of $\HA$.
   Then the smallest $k\in\fN$ such that
   $$\ThetaNcusp{k}{\sigma}\neq0$$
   satisfies that the representation $\ThetaNcusp{k}{\sigma}$ is
   cuspidal. Moreover,
   $$\ThetaNcusp{2n}{\sigma}\neq0$$
\end{thm}

\begin{proof}
   See \cite{Ral1}.
\end{proof}

\begin{rem}
   Theorem \ref{TH:Rallis} is valid for any orthogonal group, i.e. T
   need not be anisotropic.
\end{rem}


\par
So here is the plan for the rest of this section. Ultimately, the
reductive dual pair $\OurDualPair$ will be our main concern. The
non-trivial character $\OurChar$ on $\HAoverHF$ as defined in
subsection \ref{subsec:setup} defines a one dimensional automorphic
form on $\HA$. However, at this point, we will consider the family of dual
pairs $\OurDualPair[2k]$, for every $1\leq k\leq 2n$.


We will now proceed with the proof of Theorem \ref{TH:fourier}.

\begin{proof} \textit{(Theorem \ref{TH:fourier})}
We will begin by proving assertion (1) of the theorem. Note that, in
particular, assertion (1) follows from Theorem \ref{TH:Rallis}.
Nevertheless, we are going to to prove assertion (1) here. Moreover,
the proof of assertion (2) will be very similar to the proof of
assertion (1).

\newcommand{\Af}[1][\phi]{\nA^{#1}}

Consider the integral
$$\func{\Af}{g}=
\HInt{\funcIII{\thetafuncT{2n}}{g}{h}{1}\func{\xi}{h}}$$ where
$$\funcIII{\thetafuncT{2n}}{g}{h}{1}\eqdef\thetakernel{2n}{g}{h}{\bfv}$$
Note that this integral is well-defined since $\HAoverHF$ is
compact and the integrand is continuous. Recall that the
compactness of $\HAoverHF$ follows from the fact that the matrix T
is anisotropic. The compactness of $\HAoverHF$ will allow us to
exchange the order of summation and integration in what follows.


\newcommand{\Zmat}[1][2n]{\mativ{\nI_{#1}}{\nz}{}{\nI_{#1}}}
Consider the Abelian group of unipotent matrices given by
$$\nZ\eqdef\defset{g=\Zmat}{\tr{(\wn\nz)}=\wn\nz}$$
This is the unipotent radical of the Siegel parabolic subgroup of
$\tn{Sp}_{4n}$. For each $\nC~=~\tr{\nC}$, ($2n\times 2n$
symmetrical matrix with coordinates in $\F$), we associate a
character
$$\Achar[\nC]{\Zmat}=\ACHAR{C\wn z}$$

and since $\tn{tr}(C\wn Z)\in\F$  for $z\in\nZ(\F)$ and $\psi$ is
trivial on $\F$ we have that $\achar[\nC]$ is a character of
$\ZAoverZF$. Moreover, \underline{any} character of $\ZAoverZF$ has
the form $\achar$, for a uniquely defined
$\nC=\tr{\nC}\in\nM_{2n}(\F)$

\newcommand{\fourcoef}[1][\nC]{\Af_{#1,\psi_{#1}}}
\newcommand{\fourcoefphi}[2][\nC]{\Af[#2]_{#1,\psi_{#1}}}

$\Af(g)$ is a continuous function on $\GAoverGF$ and by Abelian
Fourier Analysis (see \cite{Rud}), we may expand $\Af$ in a
Fourier series along $\nZ$, since $\ZAoverZF$ is a compact,
Abelian group. For each $g\in\GA$ we consider the continuous
function on $\ZAoverZF$
$$z\mapsto\Af(\Zmat g)$$
and we now use Fourier analysis on the compact subgroup
$\ZAoverZF$.

Thus
$$\Af(\Zmat g)=\sum_{C=\tr{C}\in\MF{2n}}\Af_{C,\psi_C}(g)\Achar[C]{\Zmat}\qquad \Zmat\in\ZA,\ g\in\GA$$
where
$$\fourcoef(g)\eqdef\ZInt{\Af(\Zmat g)\invAchar[C]{\Zmat}}$$

Moreover, $\Af\equiv0$ if and only if $\fourcoef\equiv0$ for all
$\nC=\tr{\nC}$.
 We will now prove the following:
$$\Af_{\nT,\achar[T]}\not\equiv 0$$
and this will imply
$$\Af\not\equiv 0$$
(Choose $\nC=\nT$.) We will show that there is $\phi$ such that
$$\fourcoef[\nT](\nI_{4n})\not\equiv0$$
\seteqi{\Af_{\nT,\achar[T]}(\nI_{4n})}
\seteqii{\ZInt{\HInt{\thetafuncT{2n}(\Zmat,h,1)\OurChar(h)}\invACHAR{T\wn
z}}}

\seteqiii{\ZInt{\HInt{\thetakernel{2n}{\Zmat}{h}{\bfv}\OurChar(h)}\invACHAR{T\wn
z}}}

\seteqiiii{\ZInt{\HInt{\thetakerneli{2n}{\Zmat}{1}{\bfv}{1}{h}\OurChar(h)}\invACHAR{T\wn
z}}}
\seteqiiiii[\stackit{\tn{Lem}}{\ref{lem:weilformulas}}]{\ZInt{\HInt{\thetakernelii{2n}{\ACHAR{\Gr(\bfv)\wn
z}}{\bfv}{1}{h}\OurChar(h)}\invACHAR{T\wn z}}}
\seteqiiiiii{\HInt{\thetakernel{2n}{1}{h}{\bfv}\OurChar(h)}\ZInt{\ACHAR{(\Gr(\bfv)-T)\wn
z}}}
\seteqiiiiiii[(*)]{\HInt{\thetakernelj{2n}{1}{h}{\bfv}{\stackrel{\bfv\in\VF{2n}}{\Gr(\bfv)=\nT}}\OurChar(h)}}

\showeqiiiiiii{-40pt}

The last transition follows from the fact that, for
$\alpha\in\nM_{2n}(\F)$, such that $\tr{(\wn\alpha)}=\wn\alpha$
$$\ZInt{\ACHAR{\alpha z}}=\Braceii{1}{\tn{if }\alpha=0}{0}{\tn{otherwise}}$$

The other transitions are fairly straightforward. At first we simply
apply the definition of $\thetafuncT{2n}$, followed by the next
three transitions in which we use fundamental properties of
$\Weil{2n}$. However, the following transition in which we change
the order of integration and summation, needs more explanation. At
first, at least formally, we switch between $\D{h}$ and $\D{z}$,
next we switch the order of $\D{z}$ and $\Sigma_{\bfv}$. This is all
permissible since the integrals we get in each case converge
absolutely. Indeed, if we take the measure of $\ZAoverZF$ to be 1,
then we have
\seteqi{\HInt{\Sum{\bfv\in\VF{2n}}\ZInt{\abs{\funcIII{\Weil{2n}}{1}{h}{1}\func{\phi}{\bfv}\OurChar(h)\ACHAR{(\Gr(\bfv)-T)\wn
z}}}}} \seteqii[\stackit{\tn{Since
}}{\abs{\OurChar(\cdot)}=\abs{\psi(\cdot)}=1}]
{\HInt{\Sum{\bfv\in\VF{2n}}\ZInt{\abs{\funcIII{\Weil{2n}}{1}{h}{1}\func{\phi}{\bfv}}}}}
\seteqiii{\HInt{\Sum{\bfv\in\VF{2n}}\abs{\func{\phi}{\inv{h}\cdot\bfv}}}<\infty}
\showeqiii{-20pt}

which follows from the fact that
$h\mapsto\Sum{\bfv\in\VF{2n}}\abs{\func{\phi}{\inv{h}\cdot\bfv}}$
is continuous and $\HAoverHF$ is compact.

Let $\set{f_1,\ldots,f_{2n}}$ be an orthogonal basis such that
$\diag((f_1,f_1)_{\nT},\ldots,(f_{2n},f_{2n})_{\nT})~=~\nT$ for
$\nV$. Let
$$\vnot=\colvec{f_1}{f_{2n}}$$
Therefore $Gr(\vnot)=((f_i,f_j)_{\nT})_{i,j}=\diag\rowvec{\t_1}{\t_n}=\nT$,
where as one may recall, $\Gr(\vnot)$ denotes the Gram matrix of
$\vnot$.
\par
It follows from Witt's theorem that the set
$\defset{\bfv\in\VF{2n}}{\Gr(\bfv)=\nT}$ forms an orbit
under $\HF$ with a trivial stabilizer.\\
Therefore there exists a $\gamma\in\HF$ such that
$\bfv=\inv{\gamma}\cdot\vnot$
\\
Therefore the integral (*) is equal to
\seteqi{\HInt{\thetakernelj{2n}{1}{h}{\inv{\gamma}\cdot\vnot}{\gamma\in\HF}\OurChar(h)}}
\seteqii[\inbrackets{\tn{Lemma
}\ref{lem:weilformulas}}]{\HInt{\thetakernelj{2n}{1}{\gamma
h}{\vnot}{\gamma\in\HF}\OurChar(h)}}
\seteqiii{\HInt{\thetakernelj{2n}{1}{\gamma
h}{\vnot}{\gamma\in\HF}\OurChar(\gamma h)}}
\seteqiiii{\XInt{\Lmana{\HA}{\set{1}}}{\Weil{2n}(1,h,1)\phi(\vnot)\OurChar(h)}{h}}
\seteqiiiii[\inbrackets{\tn{Lemma }\ref{lem:weilformulas}}]
           {\XInt{\HA}{\phi(\inv{h}\cdot\vnot)\OurChar(h)}{h}}
\seteqiiiiii[\stackit{\tn{Assuming that }
\phi=\prod_\nu\local{\phi}}{\tn{is decomposable}}]
{\prodnu{}{\locXInt{\locHF}{\local{\phi}(\inv{h}\cdot\vnot)\locOurChar(h)}}{h}}

\showeqiiiiii{-30pt}


We show that there exists a $\phi\in\glbSchwartz$ such that
$$\prodnu{}{\locXInt{\locHF}{\local{\phi}(\inv{h}\cdot\vnot)\locOurChar(h)}}{h}\neq0$$

\newcommand{\locprod}[2]{\prodnu{#1}{\locXInt{\locHF}{\local{\phi}(\inv{h}\cdot\vnot)#2}}{h}}
\newcommand{\locprodH}[2]{\prodnu{#1}{\locXInt{\local{\nH}}{\local{\phi}(\inv{h}\cdot\vnot)#2}}{h}}

\seteqi{\locprod{}{\locOurChar(h)}}

\seteqii{\locprod{\in\nS}{\det(h)}\cdot\locprod{\not\in S}{1}}

\seteqiii[\stackit{\tn{For }\nu\not\in S}{\tn{let
}\localphi=\chi_{\nV(\locOF)}^{2n}}]{\locprod{\in\nS}{\det(h)}\cdot1}

\seteqiiii[\stackit{\tn{Choose small neighborhood
}\local{\nH}\tn{of I}}{\tn{s.t.}\det h=1, h\in\local{\nH}.\tn{Let
}\supp\localphi\IN\local{\nH}}]{\locprodH{\in\nS}{\det(h)}}

\seteqiiiii[(**)]{\locprodH{\in\placeFinf\bigcap\nS}{}\cdot\locprodH{\in(\placeF-\placeFinf)\bigcap\nS}{}}

\newcommand{\Ai}{\tn{For }\nu\in\placeFinf\bigcap\nS,\tn{ let
}\localphi\tn{ be any positive Schwartz function, supported in
}\local{\nH}}
\newcommand{\Aii}{\tn{For }\nu\in(\placeF-\placeFinf)\bigcap\nS,\tn{ let }\localphi,
   \tn{ be the characteristic function of }\local{\nH}\cdot\vnot}

\showeqiiiii{-40pt}

where in the second transition we used the fact that we may choose
the measure
$$\meas(\defset{h\in\prod_{\nu\not\in\nS}\locHF)}{\inv{\local{h}}\in\nV^{2n}(\locOF)}$$
to be equal to one.

And if we let
\begin{itemize}
  \item $\Ai$
  \item $\Aii$
\end{itemize}
then (**) is non-zero (since it is positive).

We will now prove assertion (2) of Theorem \ref{TH:fourier}. We
will begin by proving that $\fourcoef(1)=0$ for all $\nC=\tr{\nC}$
of size $k\times k$, with coordinates in $\F$ and
$\phi\in\nS(\nV_{\fA}^k), k<2n$. This will imply that for every
$g\in\GA$, we have
$$\fourcoef(g)=\fourcoefphi{\Weil{k}(g,1,1)\phi}(1)=0$$

\newcommand{\psiC}{\psi_{\nC}}

The proof of (2), coincides with the proof of (1), up to the point
where we describe the stabilizer of an element in
$\defset{v\in\nV_\F^k}{\tn{Gr}(v)=\nC}$. In (1), the stabilizer
was trivial and in (2) it won't be. This fact will imply that the
Fourier coefficient is zero, for all $1\leq k< 2n$. Thus,
repeating our previous calculations (with C replacing T, of
course) we get

\seteqi{\ZInt{\HInt{\thetafunc{k}(\Zmat,h,1)\OurChar(h)}\invACHAR{C\wn
z}}}
\seteqii{\ZInt{\HInt{\thetakernel{k}{\Zmat}{h}{\bfv}\OurChar(h)}\invACHAR{C\wn
z}}}
\seteqiii{\ZInt{\HInt{\thetakerneli{k}{\Zmat}{1}{\bfv}{1}{h}\OurChar(h)}\invACHAR{C\wn
z}}}
\seteqiiii[\stackit{\tn{Lem}}{\ref{lem:weilformulas}}]{\ZInt{\HInt{\thetakernelii{k}{\ACHAR{\Gr(\bfv)\wn
z}}{\bfv}{1}{h}\OurChar(h)}\invACHAR{C\wn z}}}
\seteqiiiii{\HInt{\thetakernel{k}{1}{h}{\bfv}\OurChar(h)}\ZInt{\ACHAR{(\Gr(\bfv)-C)\wn
z}}}
\seteqiiiiii[(**)]{\HInt{\thetakernelj{k}{1}{h}{\bfv}{\stackrel{\bfv\in\VF{k}}{\Gr(\bfv)=C}}\OurChar(h)}}

\showeqiiiiii{-20pt}

\newcommand{\HFo}{\HF^{\boldmath{v^0}}}
\newcommand{\HAo}{\HA^{\boldmath{v^0}}}
\newcommand{\vnoti}[1]{v^0_{#1}}
\newcommand{\Fvnoti}[1]{\F v^0_{#1}}

where the second to last transition in which we exchanged order of
integration/summation is justified exactly as before.
\\
It follows from Witt's theorem that the set
$\defset{\bfv\in\VF{k}}{\Gr(\bfv)=\nC}$ forms an orbit under
$\HF$.

\newcommand{\uni}[1]{\mativ{\nI_{k}}{#1}{}{\nI_{k}}}

We may assume that the symmetric matrix C is diagonal. This follows
from the fact that there is an $a\in\tn{GL}_k(\F)$ such that
$$a\cdot\nC\cdot\tr{a}=\diag(c_1,\ldots,c_k)$$
and that
$$\Af_{\achar}\mativ{\wn\tr{a}\wn}{}{}{\inv{a}}=\Af_{\achar[a\cdot\nC\cdot\tr{a}]}(\nI)$$
Thus, we may assume that
$\nC=\diag(c_1,\ldots,c_r,\underbrace{0,\ldots,0}_{k-r\tn{ times}})$
where $\rank\nC=r$, i.e. $c_1,\ldots,c_r\in\Fstr$. Then
$\Gr(\bfv)=\nC$ means that there exists
$v=(v_1,\ldots,v_k)\in\nV_\F^k$ such that
\begin{eqnarray*}
\orthform{v_i}{v_j} & = & 0\tn{ for }1\leq i\neq j\leq k \\
\orthform{v_i}{v_i} & = & 0\tn{ for }r+1\leq i\leq k \\
\orthform{v_i}{v_i} & = & c_i\tn{ for }1\leq i\leq r
\end{eqnarray*}
Since $\orthform{\cdot}{\cdot}$ is an anisotropic quadratic form,
$\orthform{v_i}{v_i}=0\iff v_i=0$. Thus
$v=(v_1,\ldots,v_r,0,\ldots,0)$.

\newcommand{\wo}[1]{w^0_{#1}}
Choose a representative
$\vnot=(\wo{1},\ldots,\wo{r},\underbrace{0,\ldots,0}_{k-r\tn{
times}})$, and let $\HFo$ denote its stabilizer in $\HF$. Then
$$\HFo=\defset{g\in\HF}{g\wo{i}=\wo{i},\quad 1\leq i\leq r}$$
Let $\nV^0\eqdef\SpanF\set{\wo{1},\ldots,\wo{r}}$ and complete
$\set{\wo{1},\ldots,\wo{r}}$ to an orthogonal basis of V:
$\set{\wo{1},\ldots,\wo{r};\wo{r+1},\ldots,\wo{2n}}$ (Its Gram
matrix will have the form
$\diag(c_1,\ldots,c_r,d_{r+1},\ldots,d_{2n})$). Let
$$\nV_0\eqdef\SpanF\set{\wo{r+1},\ldots,\wo{2n}}$$
Then $\nV=\nV^0\oplus\nV_0$ and $\HFo\cong\nO(\nV_0)$ where the
isomorphism is given by
$$g\in\HFo\mapsto g|_{\nV_0}\in\nO(\nV_0).$$

Now since $k<2n$, this implies $r \leq k < 2n$, hence
$\dimF\nV_0>1$.


Therefore, since, by its definition, $\OurChar$ is a non-trivial
character, when restricted to $\nO(\nV_0)_\fA$, we get
$$\Int{\HFo\backslash\HAo}\OurChar(r)\tn{d}r=0$$
Note also that the integral on the left-hand side makes sense
since $\HFo\backslash\HAo$ is always compact.

\par

Returning to our calculation of our Fourier coefficient we now
have
\seteqi{\HInt{\thetakernelj{2k}{1}{h}{\bfv}{\stackrel{\bfv\in\VF{2k}}{\Gr(\bfv)=C}}\OurChar(h)}}

\seteqii{\HInt{\thetakernelj{2k}{1}{h}{\inv{\gamma}\cdot\vnot}{\gamma\in\HFo\backslash\HF}\OurChar(h)}}

\seteqiii[\stackit{\tn{Lem}}{\ref{lem:weilformulas}}]{\HInt{\thetakernelj{2k}{1}{\gamma
h}{\vnot}{\gamma\in\HFo\backslash\HF}\OurChar(h)}}

\seteqiiii{\HInt{\thetakernelj{2k}{1}{\gamma
h}{\vnot}{\gamma\in\HFo\backslash\HF}\OurChar(\gamma h)}}

\seteqiiiii{\Int{\HFo\backslash\HA}\Weil{2k}(1,h,1)\phi(\vnot)\OurChar(h)\D{h}}

\seteqiiiiii{\Int{\HAo\backslash\HA}\Int{\HFo\backslash\HAo}\Weil{2k}(1,rh,1)\phi(\vnot)\OurChar(rh)\D{r}\D{h}}

\seteqiiiiiii[\stackit{\tn{Lem}}{\ref{lem:weilformulas}}]{\Int{\HAo\backslash\HA}\Int{\HFo\backslash\HAo}\Weil{2k}(1,h,1)\phi(\inv{r}\cdot\vnot)\OurChar(rh)\D{r}\D{h}}

\seteqiiiiiiii{\Int{\HAo\backslash\HA}\Int{\HFo\backslash\HAo}\Weil{2k}(1,h,1)\phi(\vnot)\OurChar(rh)\D{r}\D{h}}

\seteqix{\Int{\HAo\backslash\HA}\Weil{2k}(1,h,1)\phi(\vnot)\OurChar(h)\D{h}\Int{\HFo\backslash\HAo}\OurChar(r)\D{r}}

\seteqx{0}

\showeqx{-20pt}

\end{proof}

\begin{proof}[\textit{Proof of Theorem \ref{TH:ThetaRepCuspidal}}]
   The cuspidality of
   $\ThetaNcuspT{2n}{\OurChar}$ follows from Theorem \ref{TH:fourier} and Theorem
   \ref{TH:Rallis} (Rallis' tower property). The irreducibility of
   $\ThetaNcuspT{2n}{\OurChar}$ follows from a general theorem of
   Moeglin \cite{M} stating that if $\ThetaNcuspT{k}{\sigma}$ (for an irreducible cuspidal
   $\sigma$) is cuspidal, then $\ThetaNcuspT{k}{\OurChar}$ is
   irreducible.
\end{proof}

\subsection{Explicit Construction of our Representation} \label{subsec:explicitcon}
\newcommand{\Tform}[2]{\TT(#1,#2)}
\newcommand{\TFORM}[2]{\INPROD{\thetaintT[#1]}{#2(g)}}
\newcommand{\Vwave}{\widetilde{\nV}}

\begin{proof}[Proof of Theorem \ref{TH:ExplicitCon}]

Let
$$\pi\eqdef\ThetaNcuspT{2n}{\OurChar}$$

According to Flath (see \cite{F}), for every place $\nu$ there
exists an irreducible representation $\localpi$ of $\locGF$ such
that
$$\pi\cong\tensor_{\nu}'\localpi$$
where almost all $\localpi$ are unramified representations of
$\locGF$, i.e. $\localpi$ is a constituent of
$\Ind{\locGF}{\local{B}}{\localmu}$ where $\localmu$ is an
unramified character, which means $\localmu$ is of the form

$$\localsmu\begin{pmatrix}
  t_1    &           &        &             &        &        \\
         & \ddots    &        & \mathbf{*}  &        &        \\
         &           & t_{2n} &             &        &        \\
         &           &        & t_{2n}^{-1} &        &        \\
         & \mathbf{0}&        &             & \ddots &        \\
         &           &        &             &        & t_1^{-1}
\end{pmatrix}=\onetwon{}{\abstinu{1}} {\abstinu{2}} {\abstinu{2n}}\qquad,\isin{\subII{s}{i}{\nu}} {\fC}$$

   Let
   $$\pivarphi(g)\eqdef\thetafuncT{2n}(g)\eqdef\thetaintT$$
   be a cusp form in the space $\pi$,
   such that $\pivarphi\not\equiv0$. Such a non-zero $\pivarphi$ exists
   according to Theorem \ref{TH:fourier}(1). Therefore,
   \begin{equation*}
      0\neq\Int{\GAoverGF} \abs{\pivarphi(g)}^2dg\\
      =\INPROD{\thetaintT}{\pivarphi(g)}
   \end{equation*}

   Now an automorphic realization of the contragredient
   representation $\pihat$ is given by
   $$\pihat\cong\pibar=\set{g\mapsto\overline{\pivarphi(g)}}$$
   where the action is given by right translation.

   Thus we have a non-trivial bilinear form $\map{\TT}{\glbSchwartz\times\Vpihat}{\fC}$ given by
   $$0\not\equiv\Tform{\phi}{\overline{\pivarphi}}\eqdef\TFORM{\phi_1}{\pivarphi}\qquad(***)$$

   We will now prove the following easy lemma
   \begin{lem}\label{lem:hom}
      \begin{enumerate}
         \item The bilinear form
         $\map{\TT}{\glbSchwartz\times\Vpihat}{\fC}$ given by (***)
         is an intertwining operator in
         $\Hom{\GA\times\HA}{\Weil{2n}\tensor\pihat}{1_{\GA}\tensor\OurChar}$.
         \item
         $\Hom{\GA\times\HA}{\Weil{2n}\tensor\pihat}{1_{\GA}\tensor\OurChar}\neq0$.
         \item
         $\Hom{\locGF\times\locHF}{\locWeil{2n}\tensor\localpihat}{1_{\locGF}\tensor\locOurChar}\neq0$
         for all places $\nu$.
      \end{enumerate}

   \end{lem}

   \begin{proof}
      Proof of (2).\\
      We have already demonstrated that the bilinear form defined in
      (***) is non-zero, therefore if we prove (1), then (2)
      follows.

      Proof of (1).\\
      $$\TT\in\Hom{\HA\times\GA}{\Weil{2n}\tensor\pihat}{1_{\GA}\tensor\OurChar}$$
      $$\iff\qquad\Tform{\Weil{2n}(\gnot,\hnot,1)\phi}{\pihat(\gnot)\pivarphibar}
      =\inv{\OurChar}(\hnot)\Tform{\phi}{\pivarphibar}$$

      \seteqi{\Tform{\Weil{2n}(\gnot,\hnot,1)\phi}{\pihat(\gnot)\pivarphibar}}

      \seteqii{\INPROD{\genthetaintT[\Weil{2n}(\gnot,\hnot,1)\phi]{g}{h}{h}}{\pivarphi(g\gnot)}}

      \seteqiii{\INPROD{\genthetaintT{g\gnot}{h\hnot}{h}}{\pivarphi(g\gnot)}}

      \seteqiiii{\inv{\OurChar}(\hnot)\cdot\INPROD{\genthetaintT{g\gnot}{h\hnot}{h\hnot}}{\pivarphi(g\gnot)}}

      \seteqiiiii{\inv{\OurChar}(\hnot)\cdot\INPROD{\genthetaintT{g}{h}{h}}{\pivarphi(g)}}

      \seteqiiiiii{\inv{\OurChar}(\hnot)\Tform{\phi}{\pivarphibar}}

      \seteqiiiiiii[\inbrackets{\tn{Since }\OurChar^2=1}]{\OurChar(\hnot)\Tform{\phi}{\pivarphibar}}

      \showeqiiiiiii{-20pt}

      Proof of (3).\\
      Choose a non-zero element
      $\TT\in\Hom{\HA\times\GA}{\Weil{2n}\tensor\pihat}{1_{\GA}\tensor\OurChar}$.
      Thus, there exists $\phi\in\glbSchwartz$ and
      $\pivarphibar\in\Vpihat$ such that
      $$\Tform{\phi}{\pivarphibar}\neq0$$

      \newcommand{\Vvo}{\nV_{\nu_{0}}}
      \newcommand{\Vpihatvo}{\nV_{\pihat_{\nuo}}}
      \newcommand{\phif}{\phi_f}
      \newcommand{\localphio}{\localo{\phi}}
      \newcommand{\phinuo}{\phi_{\nuo}}

      We may assume that $\phi=\otimes'_\nu\phi_\nu$ and
      $\pivarphibar=\imath(\otimes'_\nu u_\nu)$, where $\imath$ is
      an $\GA$ embedding
      $\embedmap{\imath}{\otimes'_\nu\nV_{\local{\pihat}}}{\cA_{\tn{cusp}}(\GAoverGF)}$ such that
      $\tn{im}(\imath)=\pibar$, since such elements span $\glbSchwartz$ and $\Vpihat$, respectively. Choose a place
      $\nuo$ and let $f\in S(\Vvo^{2n})$ and $v\in\Vpihatvo$. Let $\phif\in S(\nV_\fA^{2n})$ be the
      function $\tensor\localphi'$, where $\localphi'=\localphi$ for all $\nu\neq\nuo$, and
      $\localphio'=f$. Similarly, let $\varphibar_v=\imath(u_v)$ where $u_v=\tensor u'_v$ where $u_v=u_v'$ for
      $\nu\neq\nuo$ and $u_{\nuo}'=v$. Define $\TTnuo(f,v)=\TT(\phif,\varphibar_v)$. Then $\TTnuo(\phinuo,u_{\nuo})=\TT(\phi,\varphibar)\neq0$.
      Thus $\TTnuo$ is a non-zero element of
      $$\Hom{\locGFo\times\locHFo}{\locWeilo{2n}\tensor\localpihato}{1_{\locGFo}\tensor\locOurCharo}$$
      This is true for each place $\nuo$, and the lemma is proved.
   \end{proof}

   \renewcommand{\hilbsymb}[2]{\hilbsymbI{#1}{#2}{\localF}}
   According to Lemma \ref{lem:hom}
   $$\Hom{\locGF\times\locHF}{\locWeil{2n}\tensor\localpihat}{1_{\locGF}\tensor\locOurChar}\neq0.$$
   Since
   $$\Hom{\locGF\times\locHF}{\locWeil{2n}\tensor\localpihat}{1_{\locGF}\tensor\locOurChar}\cong
   \Hom{\locGF\times\locHF}{\locWeil{2n}}{\localpi\tensor\locOurChar}$$
   we get, that for all places $\nu$,
   $$\Hom{\locGF\times\locHF}{\locWeil{2n}}{\localpi\tensor\locOurChar}\neq0.$$
   For almost all $\nu$ we have $\locOurChar\equiv1$,
   therefore for almost all $\nu$,
   $$\Hom{\locGF\times\locHF}{\locWeil{2n}}{\localpi\tensor1_{\locHF}}\neq0,$$
   which means that $\localpi$ is obtained from $1_{\locHF}$ by
   local Howe duality with respect to $\localpsi$. According to
   local Howe duality (see Section \ref{subsec:howeduality}), if
   we show
   $$\Hom{\locGF\times\locHF}{\locWeil{2n}}{\localpi'\tensor1_{\locHF}}\neq0,$$
   for $\localpi'\in\irr{\locGF}$ then
   $$\localpi\cong\localpi'.$$

   \newcommand{\locSchwartz}{\nS(\V[\localF]^{2n})}
   Recall that B denotes the Borel subgroup of $\tn{Sp}_{4n}$.
   Let $\nr(\phi)(g)=\locWeil{2n}(g,1,1)\phi(0,\ldots,0)$. Then for
   every $g\in\locGF, b\in\localB, \phi\in\nS(\V[\localF]^{2n})$ we have
   $$\nr(\phi)(bg)=\delta(b)^\HALF\muo(b)\nr(\phi)(g)$$
   where
   $$\muo=\hilbsymb{\det(a)}{(-1)^n\det(\nT)}\abslv{\det(a)}^n\delta^{-\HALF}(b)\qquad,b=\uniy\aastr$$
   and $\delta=\delta_{\localB}$.\\
   Indeed,
   \seteqi{\nr(\phi)(bg)}

   \seteqii{\locWeil{2n}(bg,1,1)\phi(0,\ldots,0)}

   \seteqiii{\locWeil{2n}(b,1,1)(\locWeil{2n}(g,1,1))\phi(0,\ldots,0)}

   \seteqiiii[\inbrackets{\tn{Lemma }\ref{lem:weilformulas}}]{\locWeil{2n}(\uniy,1,1)\locWeil{2n}(\aastr,1,1)(\locWeil{2n}(g,1,1))\phi(0,\ldots,0)}

   \seteqiiiii{\ACHAR{\tn{Gr}(0,\ldots,0)\wn y}\locWeil{2n}(\aastr,1,1)(\locWeil{2n}(g,1,1))\phi(0,\ldots,0)}

   \seteqiiiiii[\inbrackets{\tn{Lemma }\ref{lem:weilformulas}}]{\locWeil{2n}(\aastr,1,1)(\locWeil{2n}(g,1,1))\phi(0,\ldots,0)}


   \seteqiiiiiii{\hilbsymb{\det(a)}{(-1)^{\frac{2n}{2}}\det(\nT)}\abslv{\det(a)}^{\frac{2n}{2}}\locWeil{2n}(g,1,1)\phi(0a,\ldots,0a)}

   \seteqiiiiiiii{\hilbsymb{\det(a)}{(-1)^n\det(\nT)}\abslv{\det(a)}^n\delta^{-\HALF}(b)\delta^{\HALF}(b)\locWeil{2n}(g,1,1)\phi(0,\ldots,0)}

   \seteqix{\muo(b)\delta^{\HALF}(b)\locWeil{2n}(g,1,1)\phi(0,\ldots,0)}

   \seteqx{\muo(b)\delta^{\HALF}(b)\nr(\phi)(g).}

   \Showeqx

   Therefore, we have shown that
   $$\nr(\phi)\in\Ind{\locGF}{\localB}{\muo}.$$
   We will now show that
   $$\nr\in\Hom{\locGF\times\locHF}{\locWeil{2n}}{(\Ind{\locGF}{\localB}{\muo})\tensor1_{\locHF}}.$$
   We must show that the following diagram commutes
   $$\commdiagi{\locSchwartz}{\Ind{\locGF}{\localB}{\muo}}{\nr}{\nr}{\locWeil{2n}(g,h,1)}{(\rho\otimes1_{\locHF})(g,h)}$$
   where $\rho$ denotes right translation.
   We have

   \seteqi{\nr(\locWeil{2n}(g,h,1)\phi)(\go)}

   \seteqii{\locWeil{2n}(\go,1,1)\locWeil{2n}(g,h,1)\phi(0,\ldots,0)}

   \seteqiii{\locWeil{2n}(\go g,h,1)\phi(0,\ldots,0)}

   \seteqiiii[\inbrackets{\tn{Lemma }\ref{lem:weilformulas}}]{\locWeil{2n}(1,h,1)\locWeil{2n}(\go g,1,1)\phi(0,\ldots,0)}

   \seteqiiiii{\locWeil{2n}(\go g,1,1)\phi(\inv{h}0,\ldots,\inv{h}0)}

   \seteqiiiiii{\locWeil{2n}(\go g,1,1)\phi(0,\ldots,0)}

   \seteqiiiiiii{(\rho\otimes1_{\locHF}(g,h))(\nr(\phi))(\go)}

   \Showeqiiiiiii

   Let
   $$\Vwave\eqdef\nr(\glbSchwartz)\IN\Ind{\locGF}{\locBF}{\muo}.$$
   Extend $\Vwave$ to a Jordan-H\"{o}lder series such that
   $$\onen{\IN}{\set{0}}{\V[i-1]}\IN\onetwon{\IN}{\V[i](=\Vwave)}{\V[i+1]}{\Ind{\locGF}{\locBF}{\muo}}.$$
   Therefore $\V[j-1]\backslash\V[j]$ is an irreducible constituent of
   $\Ind{\locGF}{\locBF}{\muo}$ for all $j$, and we have
   $$\Hom{\locGF\times\locHF}{\locWeil{2n}}{(\V[j-1]\backslash\V[j])\tensor1_{\locHF}}\neq0\quad\tn{for some }1\leq j\leq i.$$
   Therefore, by local Howe duality, we have
   $$\localpi\cong\V[j-1]\backslash\V[j].$$
   Let us return to the description of $\muo$. Write $b\in\localB$
   as before and write
   $$a=\begin{pmatrix}
  t_1    &           & \mathbf{*} \\
         & \ddots    &            \\
         &           & t_{2n}
\end{pmatrix}.$$

   Then
 \renewcommand{\ti}[2]{\abslv{t_{#1}}^{#2}}

   \seteqi{\muo(b)}

   \seteqii{\hilbsymb{\det(a)}{(-1)^n\det(\nT)}\abslv{\det(a)}^n\delta^{-\HALF}(b)}

   \seteqiii{\hilbsymb{t_1t_2\cdots t_{2n}}{(-1)^n\t_1\t_2\cdots\t_{2n}}\abslv{t_1t_2\cdots t_{2n}}^n\delta^{-\HALF}(b)}

   \seteqiiii{\hilbsymb{t_1t_2\cdots t_{2n}}{(-1)^n\t_1\t_2\cdots\t_{2n}}\abslv{t_1t_2\cdots t_{2n}}^n(\ti{1}{4n}\cdot\ti{2}{4n-2}\cdot\ti{3}{4n-4}\cdots\ti{2n-1}{4}\ti{2n}{2})^{-\HALF}}

   \seteqiiiii{\hilbsymb{t_1t_2\cdots t_{2n}}{(-1)^n\t_1\t_2\cdots\t_{2n}}\prod_{i=1}^{2n}\ti{i}{n}\prod_{i=1}^{2n}\ti{i}{-(2n-i+1)}}

   \seteqiiiiii{\hilbsymb{t_1t_2\cdots t_{2n}}{(-1)^n\t_1\t_2\cdots\t_{2n}}\prod_{i=1}^{2n}\ti{i}{n-(2n-i+1)}}

   \seteqiiiiiii{\hilbsymb{t_1t_2\cdots t_{2n}}{(-1)^n\t_1\t_2\cdots\t_{2n}}\prod_{i=1}^{2n}\ti{i}{-n+i-1}}

   \Showeqiiiiiii

   $$\therefore\muo(b)=\prod_{i=1}^{2n}\ti{i}{-n+i-1}\cdot\hilbsymb{t_1t_2\cdots t_{2n}}{(-1)^n\det\nT}\quad\tn{for almost all }\nu$$
   therefore, for almost all $\nu$, the local component $\localpi$
   is a constituent of $\Ind{\locGF}{\localB}{\unramnu}\chinu_{\nT}$ where
    $$\s=(-n,-n+1,-n+2,\ldots,-1,0,1,2,\dots,n-3,n-2,n-1)$$
   and where $\unramnu$ and $\chinu_\nT$ are as given in the statement
   of the theorem.

%
\end{proof}

We will now prove Theorem \ref{TH:IsCAP}.
\begin{proof}[\textit{Proof of Theorem \ref{TH:IsCAP}}]
   According to Theorem \ref{TH:ThetaRepCuspidal}
   $\ThetaNcuspT{2n}{\OurChar}$ is an irreducible, automorphic,
   cuspidal representation of $\GA$ and according to Theorem
   \ref{TH:ExplicitCon}, for almost all $\nu$ the local component $\localpi$ is a constituent of
   $\Ind{\locGF}{\localB}{(\unramnu\cdot\chinu_\nT)}$ where
   $\s=(-n,-n+1,-n+2,\ldots,-1,0,1,2,\dots,n-3,n-2,n-1)$. Thus, $\pi$ is a CAP
   representation with respect to the Borel subgroup of $\GA$
   and $\tau=\unram\chi_\nT$ at almost all places $\nu$.
\end{proof}

\renewcommand{\hilbsymb}[2]{\hilbsymbI{#1}{#2}{\F}}

\newpage



\vspace{0.5cm}
\noindent
{\scshape{School of Mathematics, Tel Aviv university, Tel Aviv 69978, Israel}}
\\
\textit{E-mail address:} erezron@tauex.tau.ac.il

\end{document}